\documentclass[a4paper, reqno]{amsart}

\usepackage{natbib}

\usepackage{latexsym}
\usepackage{amssymb}
\usepackage{amsmath}
\usepackage{amsbsy}
\usepackage{amsthm}
\usepackage{txfonts}
\usepackage{pxfonts}
\usepackage{pict2e}
\usepackage{mathtools}
\usepackage{amscd}
\usepackage{helvet}
\usepackage{courier}
\usepackage[bottom]{footmisc}
\usepackage{type1cm}
\usepackage{bm}
\usepackage{calligra}
\usepackage{url}

\usepackage{multicol}           
\usepackage[bottom]{footmisc}   

\usepackage{enumitem}
\usepackage[all]{xy}

\newtheorem{defi}{Definition}[section]
\newtheorem{theo}{Theorem}[section]
\newtheorem{lem}{Lemma}[section]
\newtheorem{prop}{Proposition}[section]
\newtheorem{cor}{Corollary}[prop]

\newtheorem{con}{Conjecture}[section]

\newcommand{\Grz}{\mbox{\textbf{Grz}}}
\newcommand{\KM}{\mbox{\textbf{KM}}}
\newcommand{\KMt}{\textbf{KM}_{\tau}}
\newcommand{\Int}{\mbox{\textbf{Int}}}
\newcommand{\GL}{\mbox{\textbf{GL}}}
\newcommand{\Intau}{\textbf{Int}_{\tau}}
\newcommand{\Intaut}{\textbf{Int}_{\tau\sim}}
\newcommand{\Lan}{\mbox{$\mathcal{L}$}}
\newcommand{\Lana}{\mbox{$\mathcal{L}_{a}$}}
\newcommand{\Lantau}{\mbox{$\mathcal{L}_{\tau}$}}
\newcommand{\Lantaut}{\mbox{$\mathcal{L}_{\tau\sim}$}}
\newcommand{\Var}{\textbf{\textit{Var}}}

\newcommand{\fm}{\mathfrak{F}}

\newcommand{\de}[1]{\delta#1}

\newcommand{\Ho}{\mathbf{H}}
\newcommand{\Su}{\mathbf{S}}

\newcommand{\Pro}{\mathbf{P}}

\newcommand{\true}{\bm{\top}}

\newcommand{\fA}{\mbox{$\mathfrak{A}$}}

\newcommand{\fB}{\mbox{$\mathfrak{B}$}}
\newcommand{\fC}{\mbox{$\mathfrak{C}$}}

\newcommand{\cA}{\mbox{$\mathcal{A}$}}

\newcommand{\cD}{\mbox{$\mathcal{D}$}}
\newcommand{\cE}{\mbox{$\mathcal{E}$}}

\newcommand{\cV}{\mbox{$\mathcal{V}$}}

\newcommand{\direct}{\mbox{$\overset{\ra}{\fA}$}}

\newcommand{\bd}[1]{\mbox{$\delta[#1]$}}

\newcommand{\SA}{\mathcal{S}_{\mathfrak{A}}}
\newcommand{\SB}{\mathcal{S}_{\mathfrak{B}}}
\newcommand{\HA}[1]{\textsf{H}(#1)}

\newcommand{\dirlim}{\mbox{$\overset{\ra}{\fA}$}}

\newcommand{\ta}{\tau}
\newcommand{\om}{\mbox{$\omega$}}

\newcommand{\ome}{\mbox{$\omega$}}

\newcommand{\ve}{\vee}

\newcommand{\we}{\wedge}

\newcommand{\ra}{\rightarrow}

\newcommand{\Ra}{\Rightarrow}

\newcommand{\bn}{\mbox{$\sim$}}

\newcommand{\sbe}{\subseteq}
\newcommand{\fsub}{\Subset}
\newcommand{\suba}{\preccurlyeq}
\newcommand{\pack}{\mbox{$\vartriangleleft$}}

\newcommand{\hA}[1]{h_{\mathfrak{A}}(#1)}
\newcommand{\hB}[1]{h_{\mathfrak{B}}(#1)}

\newcommand{\on}{\boldsymbol{1}}
\newcommand{\ze}{\boldsymbol{0}}

\newcommand{\astar}{\mbox{$a^{\ast}$}}

\newcommand{\set}[2]{\{#1 ~|~#2\}}

\newcommand{\noinA}[1]{h_{\mathfrak{A}}(\overline{#1})}
\newcommand{\noinB}[1]{h_{\mathfrak{B}}(\overline{#1})}

\begin{document}
	\title[On one embedding]{On one embedding of Heyting algebras}

	\author[A. Muravitsky]{Alexei Muravitsky}
	\address{Louisiana Scholars' College\\
Northwestern State University\\
Natchitoches, U.S.A.} 
\email{alexeim@nsula.edu}

\keywords{the G\"{o}del-McKinsey-Tarski embedding theorem, the lattice of normal extensions of S4, the lattice of intermediate logics,
	Heyting algebra, S4-algebra}
\subjclass[2010]{Primary: 03B45, Secondary: 03B55, 03G10}	
	
	\maketitle



\begin{abstract}
		The paper is devoted to an algebraic interpretation of Kuznetsov's theorem which establishes the assertoric equipollence of intuitionistic and proof-intuitionistic propositional calculi. Given a Heyting algebra, we define an enrichable Heyting algebra, in which the former algebra is embedded. Moreover, we show that both algebras generate one and the same variety of Heyting algebras. This algebraic result is equivalent to the Kuznetsov theorem. The proposed construction of the enrichable ``counterpart'' of a given Heyting algebra allows one to observe that some properties of the original algebra are preserved by this embedding in the counterpart algebra.
\end{abstract}


\medskip

\section{The history of one question}
The present paper is about an algebraic interpretation of the main theorem of~\cite{kuz85b} by Alexander Kuznetsov. This theorem reads:
\begin{equation*}
\Int + A\vdash B ~\Longleftrightarrow ~\KM + A\vdash
B, \tag{\emph{Kuznetsov's Theorem}}
\end{equation*}
where $\Int$ and $\KM$ are intuitionistic propositional calculus and proof-intuitionistic calculus, respectively, and $A$ and $B$ are assertoric (i.e. modality-free) propositional formulas. This might seem not very impressive,
if we would not know that Kuznetsov's Theorem was one of the two key properties which helped establish~\cite{km86} that the lattices of the normal extensions of $\Grz$ (Grzegorczyk logic), of $\GL$ (provability logic) and of the aforementioned logics $\Int$ and $\KM$ are connected by the following commutative diagram:
\[
\xymatrix{
	\text{NE}\KM \ar@<1ex>[r]^\tau \ar@<1ex>[r];[]^\rho
	\ar[d]_{\lambda} &\text{NE}\GL\ar[d]^{\mu}\\
	\text{NE}\Int \ar@<1ex>[r]^\sigma  \ar@<1ex>[r];[]^{\sigma^{-1}} &\text{NE}\Grz	
}
\]
where $\tau$ and $\rho$ are lattice isomorphisms and inverses of one another,\footnote{The definitions of $\tau$ and $\rho$ can be found in~\cite{mur85} or in~\cite{mur14a}, section 7.4.8.} $\lambda$ and $\mu$ are meet epimorphisms\footnote{See definitions in~\cite{km86} or in~\cite{mur14a}, section 7.4.8.} and $\sigma$ denotes the Blok-Esakia isomorphism. Kuznetsov's Theorem is responsible for $\lambda$ to be a meet epimorphism which makes the above diagram commute. If one seeks to find a relationship between modal propositional systems on classical and intuitionistic bases, a diagram like that, we believe, gives a right view.\footnote{Kuznetsov's Theorem was generalized in~\cite{mur15b}, Proposition 4.2. As a consequence of this generalization, the above diagram has recently been extended; cf.~\cite{mur15a}.} 
\subsection{Syntactic background}\label{S:syntactic}
We will, at first, be dealing with \textit{formulas} (alias \textit{terms}) of two propositional languages, $\Lana$ and $\Lan_{\square}$. The language $\Lana$ is grounded on a denumerable set $\Var$ of propositional variables and the logical constants: $\wedge,\vee,\ra$ and $\neg$. Unspecified $\Lana$-formulas will be denoted by $A,~B,\ldots$ We obtain $\Lan_{\square}$ by adding modality $\square$ to the logical constants of $\Lana$.  Regarding the sets of $\Lana$- and $\Lan_{\square}$-formulas as algebras, we obtain the formula algebras $\fm_a$ and $\fm_{\square}$, respectively.  In Section~\ref{S:localization-proof-theoretic}, we will introduce two more extensions of $\Lana$.
Let $\Lan$ be a propositional language which is an extension (not necessarily proper) of $\Lana$. A homomorphism of a formula algebra $\fm_{\mathcal{L}}$ into $\fm_{\mathcal{L}}$ is called a \textit{substitution}. Metavariables for $\Lan$-formulas (or $\Lan$-terms) are $\alpha, \beta,\ldots$. As usual, we denote:
\[
\alpha\leftrightarrow\beta:=(\alpha\rightarrow\beta)\wedge(\beta\rightarrow\alpha).
\]

Now $\Int$ can be defined as an $\Lana$-system given by any suitable axioms for intuitionistic propositional logic
(see, e.g.,~\cite{church1956}, {\S} 26) and two inference rules, (uniform) substitution and modus ponens. $\KM$ is defined as an $\Lan_{\square}$-system by the axioms and inference rules of $\Int$ plus the following three formulas:
\begin{equation}\label{E:KM-axioms}
p\ra\square p, ~~~(\square p\ra p)\ra p, ~~~\square p\ra(q\vee(q\ra p)),
\end{equation}
where $p$ and $q$ are two distinct variables of $\Var$.\footnote{This axiomatization of $\KM$ differs from the original one; see~\cite{mur14a}, section 7.4.1.}

\subsection{Semantic background}
We assume that the reader is familiar with the notion of Heyting algebra and with the basic properties of those algebras. (See, e.g.,~\cite{rs70}, where those algebras are call pseudo-Boolean.)
\begin{defi}[$\KM$-algebra]\label{D:KM-algebra}
	An algebra $\fA=\langle\cA,\wedge,\vee,\ra,\square,\ze,\on\rangle$ is said to be a {\em$\KM$}-algebra if $\langle\cA,\wedge,\vee,\ra,\ze,\on\rangle$ is a Heyting algebra (the \textbf{Heyting reduct} of $\fA$), with a least and greatest elements $\ze$ (the zero) and $\on$ (the unit), respectively,
	and the unary operation $\square$ satisfies the following conditions:
	{\em\[
	\begin{array}{cl}
	(\text{a}) & x\le \square x,\\
	(\text{b}) & \square x\ra x\le x,\\
	(\text{c}) & \square x\le y\vee(y\ra x).
	\end{array}
	\]}
	The universe $\cA$ will often be denoted by $|\fA|$.
\end{defi}

As was noted in~\cite{mur90}, if $\fA$ is a $\KM$-algebra, the operation $\square$ is defined in $\fA$ uniquely. This gives rise to the following definition.
\begin{defi}[enrichable element, enrichable Heyting algebra]\label{D:enrichable}
	Let $\fA$ be a Heyting algebra and $a,b\in|\fA|$. We say that $b$ enriches $a$ or $a$ is enriched with $b$ if the following conditions are satisfied in $\fA$:
	{\em\[
	\begin{array}{cl}
	(\text{a}) & a\le b,\\
	(\text{b}) & b\ra a=a,\\
	(\text{c}) & b\le x\vee(x\ra a),~\textit{for any $x\in|\fA|$}.
	\end{array}
	\]}
	A Heyting algebra is called enrichable if every element of this algebra is enrichable.
\end{defi}
\begin{prop}\label{P:unique}
	Let $\fA$ be a Heyting algebra and $a\in|\fA|$. If $a$ is enrichable in $\fA$, it can be enriched with only one element.
\end{prop}
\begin{proof}
	For contradiction, we suppose that elements $b$ and $b^\prime$ enrich $a$, that is the properties (\text{a})--(\text{c}) of Definition~\ref{D:enrichable} are true for $b$ and $b^\prime$.
	Then we obtain:
	\[
	b\leq b^{\prime}\ve (b^{\prime}\ra a)=b^{\prime}\ve a=b^{\prime}.
	\]
	
	Similarly, we get $b^{\prime}\leq b$.
\end{proof}

The next observation uses the notion of a dense element of Heyting algebra; see definition, e.g., in~\cite{rs70}, chapter IV, {\S} 5.
\begin{prop}\label{P:negation} If $a$ is enriched in
	{\fA} with $b$, then the latter is dense; that is $\neg b=\ze$.
\end{prop}
\begin{proof} Indeed,
	\[
	\begin{array}{rl}
	\neg b\leq b\ra a & \Leftrightarrow \neg b\leq a\\
	& \Rightarrow \neg b\leq b\\
	& \Leftrightarrow \neg b\leq\ze.
	\end{array}
	\]
\end{proof}

We note that not all Heyting algebras are enrichable. For instance, the least element of a chain of type $1+\omega^{\ast}$ is not enrichable. However, every finite Heyting algebra is enrichable; cf.~\cite{mur14a}, Proposition 15. Thus any variety of Heyting algebras contains enrichable members.

It follows directly from Kuznetsov's Theorem that any variety of Heyting algebras is generated by those algebras in it that are enrichable. (Cf.~\cite{kuz85b}, Corollary 1.) In fact, the last statement is equivalent to Kuznetsov's Theorem.\footnote{The two are equivalent not just because both are true, but deductively equivalent in a higher order logic.} However, Kuznetsov states another equivalent of his Theorem:
\begin{quote}
	Every Heyting algebra $\fA$ is a subalgebra (up to isomorphism) of some enrichable Heyting algebra in the variety generated by $\fA$.\\ (Cf.~\cite{kuz85b}, Corollary 2.\label{Kuznetsov's-Corollary-2}) 
\end{quote}
The last observation was pointed out to Kuznetsov by the author and its proof can be found, e.g., in~\cite{mur08}, Remark 3. Thus, according to this observation, for any Heyting algebra $\fA$, there is an enrichable Heyting algebra $\fB$ such that
\begin{center}
	\begin{tabular}{cl}
		(A) &$\fA$ is embedded into $\fB$ and\\
		(B) &$\fA$ and $\fB$ generate one and the same variety.
	\end{tabular}
\end{center}

What do we know about $\fB$, besides its existence? According to Remark 3 of~\cite{mur08}, if $K$ is the class of all enrichable algebras of the variety generated by $\fA$, then $\fA\in\Su\Ho\Pro(K)$. Grounding only on the last membership, that is, not having any transparent algebraic construction of $\fB$, it is hardly possible to answer some natural questions about properties which can be preserved in $\fB$. For instance, grounding only on this membership, we do not know whether $\fB$ can be countable, providing that $\fA$ is; or whether $\fB$ can be subdirectly irreducible, if $\fA$ is. 

In the remaining part of the paper, we show, given a Heyting algebra $\fA$, how to define such $\fB$ that the properties (\text{A})--(\text{B}) are fulfilled. In fact, one possible candidate for $\fB$ has already been proposed in~\cite{mur86}, where we constructed algebra $\direct$ (see definition in Section~\ref{S:algebraLimitA}) which possesses the property (A).\footnote{The algebra $\direct$ was employed in~\cite{mur86} to prove the separation property for the proof-intuitionistic calculus. Also, this algebra was used in our proof of  the interpolation property for $\KM$; cf.~\cite{mur14a}, Section 7.4.7.}\\

Let $\Lan$ be an extension of $\Lana$. Given $\Lan$-algebras  {\fA} and {\fB}, we write
\[
\fA\suba\fB
\]
if {\fA}  is a subalgebra (up to isomorphism) of  {\fB}.

We conclude this section with the following definition.
\begin{defi}[valuation, logic of algebra]
Let $\Lan$ be a propositional language which is an extension of $\Lana$ and $\fA$ be an $\Lan$-expansion of Heyting algebra. Any homomorphism $v:\fm_{\mathcal{L}}\ra\fA$ is called a valuation (in $\fA$). The logic of algebra $\fA$ is the set
\[
L(\fA)=\set{A\in\fm_{\mathcal{L}}}{v(A)=\on,~\text{for any valuation $v$ in $\fA$}}.
\]
\end{defi}
Given a language $\Lan$, an $\Lan$-algebra $\fA$ and any nonempty set $\Gamma$ of $\Lan$-formulas, we denote
\[
\fA\models\Gamma
\]
if $\Gamma\sbe L(\fA)$. And if $A$ is an $\Lan$-formula, we write
\[
\Gamma\not\models A,
\]
if there is an $\Lan$-algebra $\fA$ such that $\fA\models\Gamma$ and $\fA\not\models A$. Finally, we use
\[
\Gamma\models A
\]
in the usual sense:
\[
\fA\models\Gamma\Longrightarrow\fA\models A.
\]
\subsection{The structure of the present paper}
In Section~\ref{S:algebraLimitA}, given a Heyting algebra $\fA$, we define an algebra $\direct$ which will play in the sequel the role of $\fB$ in the above conditions (A) and (B). We show (referring chiefly to~\cite{mur86}) that $\direct$ satisfies (A). Also, we demonstrate  some preservation properties over transition from $\fA$ to $\dirlim$ and state the main theorem (Theorem~\ref{T:main-theorem}). In Section~\ref{S:reducibility}, we find a sufficient condition (Corollary~\ref{C:reducibility-4}) for the main theorem. This leads us to the idea of \textit{one-element enrichment}. In Section~\ref{S:localization-algebraic} we develop an algebraic view on one-element enrichment and in Section~\ref{S:localization-proof-theoretic-0} a proof-theoretic view on it. In Section~\ref{S:connecting-viewpoints} we connect these viewpoints; in the end of that section we explain what remains to be done to complete the proof of the main theorem. We are taking a decisive step in our proof in Section~\ref{S:Stone-embedding-properties-main}.
In Section~\ref{S:completing-proof} we make a final effort to complete the proof of Theorem~\ref{T:main-theorem}. Thus, as the reader can see, this paper is devoted to the proof of one theorem. In Section~\ref{S:conclusion} we formulate open questions about properties which $\direct$ may have, providing that $\fA$ possesses them.

\section{Algebra $\direct$}\label{S:algebraLimitA}
In this paper we deal mostly with Heyting algebras or algebras whose assertoric reduct is a Heyting algebra. When confusion is unlikely, the word ``Heyting'' will often be omitted. The following algebraic notions and facts  will be presupposed. The main references in this section are \cite{rs70},~\cite{gra79},~\cite{gor98}, and \cite{mur86}. We start with the following notions and notations.

\begin{itemize}
	\item Given a Heyting algebra $\fA$ and $X\subseteq|\fA|$, we denote by $[X)_{\mathfrak{A}}$ the filter of $\fA$ generated by the set $X$; it is well know that $[X)_{\mathfrak{A}}=\set{x}{y\le x,~\text{for some $y\in X$}}$.
	\item Given a Heyting algebra $\fA$, $\SA$ denotes both the set of all prime filters of $\fA$ and the poset $(\SA,\sbe)$.\footnote{We avoid
		the term \emph{Stone space} in this paper, because topology plays no part in
		our consideration.} The filters of an algebra $\fA$ are called $\fA$-\textbf{\textit{filters}}.
	\item Given an algebra $\fA$, $\HA{\SA}$ is the (Heyting) algebra of all upward sets of $\SA$. It is well known that
	the signature operations of $\HA{\SA}$ are defined as follows:
	\[
	\begin{array}{l}
	X\we Y=X\cap Y,\\
	X\ve Y=X\cup Y,\\
	X\ra Y=\set{F\in\SA}{\forall F^{\prime}\in\SA.~(F\sbe
		F^{\prime}\&F^{\prime}\in X)\Ra
		F^{\prime}\in Y},\\
	\neg X=X\ra\emptyset,\\
	\ze=\emptyset~\text{and}~\on=\SA.
	\end{array}
	\]
	\item Given an algebra $\fA$, the Stone embedding $h_{\mathfrak{A}}: \fA\ra \HA{\SA}$  is defined as follows: $h_{\mathfrak{A}}(x)=\set{F\in\SA}{x\in F}$. The isomorphic image of $\fA$ w.r.t. $h_{\mathfrak{A}}$ is denoted by $h_{\mathfrak{A}}[\fA]$. Thus $h_{\mathfrak{A}}[\fA]\suba\HA{\SA}$. Also, we denote
	\[
	h_{\mathfrak{A}}(\overline{x})=\set{F\in\SA}{x\not\in F},
	\]
	for any $x\in|\fA|$.
	Further, we will find the usefulness of the set
	\[
	\max\noinA{x},
	\]
	which consists of the maximal elements, if any, of the set $\noinA{x}$.\\
	
\end{itemize}

Now we will outline an algebraic construction defined in~\cite{mur86},  \S1.

We start with the definition of operation $\de$ ~in $\HA{\SA}$:  
\[
\de{X}=\set{F\in\SA}{(\forall F^{\prime}\in\SA)(F\subset
	F^{\prime}\Ra F^{\prime}\in X)}.
\]
In particular, for any particular $a\in\fA$,
\begin{equation}\label{E:delta-for-element}
\de{\hA{a}}=\set{F\in\SA}{(\forall F^{\prime}\in\SA)(F\subset
	F^{\prime}\Ra a\in F^{\prime})}.
\end{equation}

In the sequel, we will need the following two observations.
\begin{prop}[\cite{mur86}, Lemma 1]\label{P:delta-h(x)-1}
Let $\fA$ be a Heyting algebra. For any $x\in|\fA|$, $\de{\hA{x}}=\hA{x}\cup\max\noinA{x}$.	
\end{prop}
\begin{prop}[\cite{mur86}, Lemmas 3 and 4]\label{P:delta-h(x)-2}
	For any $x\in|\fA|$, the element $h_{\mathfrak{A}}(x)$ is enriched in {\em $\HA{\SA}$} with the element $\de{h_{\mathfrak{A}}(x)}$. Moreover, 
	if an element $a\in|\fA|$ is enriched in {\fA} with an element $\astar$,
	then $\de{h_{\mathfrak{A}}(a)}=h(\astar)$.
\end{prop}

We will be using the following notation:
\[
h_{\mathfrak{A}}(\fA)=\set{h_{\mathfrak{A}}(x)}{x\in|\fA|}
~\text{and}~\Delta_{\mathfrak{A}}=\set{\de{h_{\mathfrak{A}}(x)}}{x\in|\fA|}.
\]
Then, we denote by $\bd{\fA}$ the subalgebra of $\HA{\SA}$ generated by $h_{\mathfrak{A}}(\fA)\cup\Delta_{\mathfrak{A}}$.\\

Next, given an algebra {\fA}, we define a denumerable sequence of
algebras as follows:
\[
\begin{array}{l}
\fA_{0}=\fA,\\
\fA_{i+1}=\bd{\fA_{i}}\quad (i<\om).\\
\end{array}
\]

Along with the sequence $\{\fA_{i}\}_{i<\om}$, we also have the
embeddings:
\[
\begin{array}{l}
\varphi_{ii}:\fA_{i}\ra\fA_{i}, i<\om,\quad\text{(the identity
	embedding
	on $\fA_{i}$)}\\
\varphi_{i(i+1)}:\fA_{i}\ra\fA_{i+1},i<\om,\quad\text{(Stone
	embedding
	$h_{\mathfrak{A}_i}:\fA_{i}\ra\bd{\fA_{i}}$)}\\
\varphi_{ij}=\varphi_{i(i+1)}\circ\varphi_{(i+1)(i+2)}\circ\ldots\circ\varphi_{(j-1)j},
\text{ where $i<j$}.
\end{array}
\]

Thus the sequence $\{\fA_{i}\}_{i<\om}$ along with the embeddings
$\varphi_{ij}$, $i\leq j$, form a direct family~\cite{gra79}. Let
{\dirlim} be the direct limit of this family.

We remind the reader that the carrier of {\dirlim} consists of the
equivalence classes on $\bigcup\set{\cA_i}{i<\ome}$:
\[
|x|=\set{y}{y\equiv x},
\]
for any $x\in\bigcup\set{\cA_i}{i<\ome}$. Here the equivalence
$x\equiv y$, where $x\in\cA_i$ and $y\in\cA_j$, means that either
$i\leq j$ and $\varphi_{ij}(x)=y$, or $i\geq j$ and
$\varphi_{ji}(y)=x$.

Obviously, for any $x,y\in\cA_i$,
\begin{equation}\label{E:equality}
|x|=|y|~\Longleftrightarrow~x=y.
\end{equation}

Next we define:
\[
|x|\odot|y|=|\varphi_{ij}(x)\odot y|,
\]
where $\odot\in\{\we,\ve,\ra\}$, $x\in\fA_i$, $y\in\fA_j$, and
$i\leq j$. (In case $j\leq i$, we define
$|x|\odot|y|=|x\odot\varphi_{ji}(y)|$.)

Naturally, we also define:
\[
\neg|x|=|\neg x|.
\]
(Cf.~\cite{gra79} or \cite{gor98}.)

Further, we define: For $x\in\cA_i$ and $y\in\cA_j$,
\[
|x|\leq|y|~\Longleftrightarrow~\text{either $i\le j$
and $\varphi_{ij}(x)\leq_j y$, or $j\le i$ and $x\le_i\varphi_{ji}(y)$},
\]
where $\le_i$ and $\le_J$ are the lattice partial orderings in the algebras $\fA_i$ and $\fA_j$, respectively.

It is easy to  see that
\begin{center}
	$|x|\leq|y|$ \emph{is the lattice partial order in} {\dirlim}.
\end{center}

If we denote the unit and zero of $\fA_i$ by $\on_i$ and $\ze_i$,
respectively, then
\begin{center}
	$\set{\on_i}{i<\ome}$ \emph{is the unit of} {\dirlim}
\end{center}
and
\begin{center}
	$\set{\ze_i}{i<\ome}$ \emph{is the zero of} {\dirlim}
\end{center}

Indeed, it is obvious that $\on_i\equiv\on_j$ and
$\ze_i\equiv\ze_j$. Thus $|\ze_0|\leq|x|\leq|\on_0|$.

Since each $\fA_i$ is a Heyting algebra, we arrive at the first observation.

\begin{prop}
	{\dirlim} is a Heyting algebra.
\end{prop}
\emph{Proof}~follows from the definition of {\dirlim} and the fact that
each $\varphi_{ij}$ is an embedding. Also, we have to use
(\ref{E:equality}).\footnote{This is also a consequence of a more general property:  Any variety is closed under
	formation of direct limits; cf.~\cite{gor98}, Theorem 1.2.9. In Section~\ref{S:reducibility} we will refer to this Theorem again.}
\\

\begin{prop}\label{P:A_i-embedding}
	Each~$\fA_i$ is embedded into {\dirlim}.
\end{prop}
\begin{proof}
	It is clear that the map
	\[
	\varphi_i:x\mapsto|x|,
	\]
	where $x\in\fA_i$ and $i<\ome$, is an embedding of $\fA_i$ into
	{\dirlim}. Indeed, $\varphi_i(x)=|\on_i|$ means $|x|=|\on_i|$. Then
	we apply (\ref{E:equality}).
\end{proof}

Also, we observe the following.

\begin{prop}\label{P:countable}
	If~{\fA} is countable, then {\dirlim} is countable as well.
\end{prop}
\begin{proof}
	Since each $\fA_i$ is countable, {\dirlim} is also countable.
\end{proof}

\begin{prop}\label{P:sub-irr}
	If {\fA} is subdirectly irreducible, so are each $\fA_i$ and
	{\dirlim}.
\end{prop}
\begin{proof}
	Let $\ome$ be the pre-top element of {\fA}. Algebra $\HA{\SA}$
	is subdirectly irreducible, for $h(\ome)$ is a pre-top element in
	it. Therefore, $\bd{\fA}$ is subdirectly irreducible. By induction,
	we conclude that each $\fA_i$ is subdirectly irreducible as well.
	
	To continue,  we first observe that $\varphi_{0i}(\ome)$ is the
	pre-top element of $\fA_i$. We denote the latter element by
	$\ome_i$.
	
	Next assume that $|\ome|\leq|x|$ and $x\in\fA_i$. Then
	$\varphi_{0i}(\ome)\leq x$, which implies that either
	$x=\varphi_{0i}(\ome)$ or $x=\on_i$. Therefore,
	$\set{\ome_i}{i<\ome}$ is a pre-top element of {\dirlim}.
\end{proof}

We want to show that {\dirlim} is enrichable. We will do it by employing the following lemma.

\begin{lem}[\cite{mur86}, Corollary 2]\label{L:1}
	If an element $x\in|\fA_i|$ is enriched with an element $y\in|\fA_i|$,
	then $|x|$ is enriched with $|y|$ in {\dirlim}.
\end{lem}

\begin{prop}\label{P:directA-enrichable}
	Algebra {\dirlim} is enrichable.
\end{prop}
\begin{proof}
	Let $x\in|\fA_i|$. Then $h_{\mathfrak{A}_i}(x)$ is enriched with $\de h_{\mathfrak{A}_i}(x)$ in
	$\fA_{i+1}$. It remains to apply Lemma~\ref{L:1}.
\end{proof}

In this paper we aim to prove the following theorem.
\begin{theo}\label{T:main-theorem}
	Given a Heyting algebra $\fA$, the algebras $\fA$ and $\direct$ generate one and the same variety. In other words, $\fA$ and $\direct$ determine one and the same equational theory, that is $L(\fA)=L(\direct)$.
\end{theo}
\section{Reduction to one-element enrichment}\label{S:reducibility}
The sense of the term \textit{one-element enrichment} should become clear at the end of this section.
\begin{prop}\label{P:reducibility-1}
	Given an algebra $\fA$, the following conditions are equivalent$:$
\[
\begin{array}{cl}
(\emph{a}) &L(\fA)=L(\direct);\\
(\emph{b}) &L(\fA_i)=L(\fA_{i+1}),~\text{for all $i\ge 0$}.
\end{array}
\]
\end{prop}
\begin{proof}
	Suppose (a) is true. Since $\fA\suba\fA_i\suba\direct$, we get (b). Now assume that (b). Then each $\fA_i$ generates one and the same variety. By virtue of~\cite{gor98}, Theorem 1.2.9, $\direct$ is a subalgebra of ultraproduct of some of $\fA_i$'s and hence generates the same variety.
\end{proof}
\begin{cor}\label{C:reducibility-2}
	A sufficient condition for Theorem~\ref{T:main-theorem}
	is that for any Heyting algebra $\fA$, $L(\fA)=L(\bd{\fA})$.
\end{cor}
\begin{proof}
Indeed, if 	$L(\fA)=L(\bd{\fA})$, for any algebra $\fA$, then starting from an algebra $\fA=\fA_0$, we obtain the condition (b) of Proposition~\ref{P:reducibility-1}.
\end{proof}

In the sequel, we will be using the following notation.
\begin{itemize}
	\item Given two sets $X$ and $Y$,
	\[
	X\fsub Y
	\]
	denotes that $X$ is a finite subset of $Y$.
\item Let $\fA$ be an algebra and $X\fsub |\fA|$. We denote by $\de{[\fA_X]}$ and by
$\de{[\fA_a]}$, if $X=\lbrace a\rbrace$, the subalgebra of
$\de{[\fA]}$ generated by $|h_{\mathfrak{A}}(\fA)|\cup\set{\de{h_{\mathfrak{A}}(x)}}{x\in X}$. 
\end{itemize}
\begin{prop}\label{P:reducubility-3}
	Given an algebra $\fA$, if for any $X,Y\fsub|\fA|$, $L(\de{[\fA_X]})=L(\de{[\fA_Y]})$, then $L(\fA)=L(\bd{\fA})$ and, hence, $L(\fA)=L(\direct)$.
\end{prop}
\begin{proof}
	We notice that $\left(\de{[\fA_X]}\right)_{X\Subset|\mathfrak{A}|}$ along with identity maps is a directed family, the direct limit of which is $\de{[\fA]}$.\footnote{Compare with~\cite{gra79}, \S21, Lemma 3.} Thus, by virtue of~\cite{gor98}, Theorem 1.2.9, $L(\de{[\fA]})=L(\fA)$ (since $\fA=\de{[\fA_\emptyset]}$). Then, we apply Corollary~\ref{C:reducibility-2}.
\end{proof}
\begin{cor}\label{C:reducibility-4}
	A sufficient condition for Theorem~\ref{T:main-theorem} is that, given a Heyting algebra $\fA$, for any $a\in|\fA|$, $L(\fA)=L(\de{[\fA_a]})$. 
\end{cor}
\begin{proof}
	Suppose for any algebra $\fA$ and any $a\in|\fA|$, $L(\fA)=L(\de{[\fA_a]})$.
Let $a,b\in|\fA|$. By virtue of~\cite{mur90}, Lemma 5, the algebras $\de{[\fA_{\lbrace a,b\rbrace}]}$ and $\de{[\de{[\fA_a]}_{h(b)}]}$ are isomorphic. This implies that $L(\de{[\fA_a]})=L(\de{[\fA_{\lbrace a,b\rbrace}]})$. By induction, we conclude that for any $X\fsub|\fA|$, $L(\fA)=L(\de{[\fA_X]})$. It remains to apply Proposition~\ref{P:reducubility-3}.	
\end{proof}

In the next section, we show that the enrichabilty of an element $a$ of an algebra $\fA$ is equivalent to the existence of a unary operation associated with $a$. Unlike the unary operation $\square$ of Definition~\ref{D:KM-algebra} which ensures the enrichabilty of all elements of algebra $\fA$, the new operation associated with $a$ guarantees the enrichabilty of just $a$. We call this treatment of one-element enrichment \emph{algebraic}.

\section{One-element enrichment from an algebraic viewpoint}\label{S:localization-algebraic}
In this section, we will treat each pair
$(a,\astar)$, where $\astar$ enriches $a$, as an element of a binary
relation. The main reference in this section is~\cite{gra79}, {\S}13 and {\S}28.

\begin{defi}[{\cE}-pair, relation {\cE}]
	Given an algebra {\fA} and $a,\astar\in|\fA|$, $(a,\astar)$ is an
	\cE-pair $($in {\fA}$)$ if $a$ is enriched with $\astar$ in {\fA}.
	Then, we define\emph{:}
	\[
	\cE_{\mathfrak{A}}=\set{(a,\astar)}{(a,\astar)\mbox{ is an \cE-pair in \fA}}.
	\]
	We will drop the subscript `$\fA$' and write simply `$\cE$' when confusion is unlikely. 
\end{defi}

We note that for any Heyting algebra, its relation {\cE} is never
empty, for $(\on,\on)$ is an \cE-pair. Also, if $\om$ is the pre-top
element of a subdirectly irreducible algebra, then $(\om,\on)$ is an
\cE-pair in this algebra.
\begin{defi}[\bn-negation]\label{D:bn-negation}
	A unary operation $\bn x$ in a Heyting algebra is called tilde-negation $($or $\sim$-negation for short$)$ if the following identities hold:
	{\em\[
	\begin{array}{cl}
	(\text{a}) &x\ra y\le \bn y\ra\bn x;\\
	(\text{b}) &x\we\bn x\le\bn\on;\\
	(\text{c}) &\bn\ze\le x\ve\bn x;\\
	(\text{d}) &\bn\ze\ra\bn\on\le\bn\on,~\textit{or equivalently}~\bn\ze\ra\bn\on=\bn\on.
	\end{array}
	\]}
Sometimes, it will be convenient, instead of $\bn x$, to  write $t(x)$ (perhaps with a subscript at `$\bn$' and `$t$').	
\end{defi}

It is obvious that in any Heyting algebra with $\bn$-negation, the following quasi-identity holds:
\begin{equation}\label{E:tilde-quasi-idenity}
x\le y \Longrightarrow \bn y\le\bn x.
\end{equation}

Before we show how a \bn-negation can be defined in a Heyting algebra, we will prove some properties of this operation.
\begin{prop}\label{P:bn-properties}
	The following properties hold in any Heyting algebra with \bn-negation.
	\[
	\begin{array}{cl}
	(\emph{a}) & \bn\on\le\bn x\le\bn\ze;\\
	(\emph{b}) & \bn x\we\bn\bn x=\bn\on;\\
	(\emph{c}) & \bn x\ve\bn\bn x=\bn\ze;\\
	(\emph{d}) & \bn\bn\ze=\bn\on;\\
	(\emph{e}) & \bn\bn\on=\bn\ze;\\
	(\emph{f}) & \bn x\leftrightarrow
	\bn\bn x=\bn\on;\\
	(\emph{g}) & x\we\bn x\le\bn\bn x\le x\ve\bn x;\\
	(\emph{h}) &\bn\ze\le x\Longrightarrow\bn x=\bn\on;\\
	(\emph{i}) & \bn x=(x\ra\bn\on)\we\bn\ze;\\
	(\emph{j}) & \bn(x\ve y)=\bn x\we\bn y;\\
	(\emph{k}) & ([\bn\on,\bn\ze],\we,\ve,\bn)~\text{is
		a Boolean algebra with complementation \bn};\\
	(\emph{l}) & \bn\bn\bn x=\bn x.
	\end{array}
	\]
\end{prop}
\begin{proof}
	(\text{a}): Since $x\ra\on\le\bn\on\ra\bn x$, we derive $\bn\on\le\bn x$. On the other hand, beginning with
	$\ze\ra\bn x\le \bn x\ra\bn\ze$, we obtain $\bn x\le\bn\ze$.
	
	(b): From (a) just proved we have:
	$\bn\on\le\bn x$ and $\bn\on\le\bn\bn x$ and hence $\bn\on\le\bn x\we\bn\bn x$. And by virtue of Definition~\ref{D:bn-negation}.b, we get $\bn x\we\bn\bn x=\bn\on$.
	
	(c): According to (a) above, $\bn x\le\bn\ze$ and $\bn\bn x\le\bn\ze$ and hence $\bn x\ve\bn\bn x\le\bn\ze$. Then, with help of Definition~\ref{D:bn-negation}.c, we get
	$\bn x\ve\bn\bn x=\bn\ze$.
	
	(d): Using (a) above twice and, then, (b), we obtain:
	\[
	\bn\on\le\bn\bn\ze=\bn\ze\we\bn\bn\ze=\bn\on.
	\]
	
	(e): We use (c) and (a) twice to obtain:
	\[
	\bn\ze=\bn\on\ve\bn\bn\on\le\bn\bn\on\le\bn\ze.
	\]
	
	(f): Using (c) and (b) above and Definition~\ref{D:bn-negation}.d, we get:
	\[
	\begin{array}{rl}
	\bn x\leftrightarrow\bn\bn x
	&=(\bn x\rightarrow\bn\bn x)\we(\bn\bn x\rightarrow\bn x)\\
	&=(\bn x\rightarrow\bn x\we\bn\bn x)
	\we(\bn\bn x\rightarrow\bn x\we\bn\bn x)\\
	&=(\bn x\ve\bn\bn x)\rightarrow
	(\bn x\we\bn\bn x)\\
	&=\bn\ze\rightarrow\bn\on~~~[\text{(c) and (b)}]\\
	&=\bn\on.~~~[\text{Definition~\ref{D:bn-negation}.d}]
	\end{array}
	\]
	
	(g): We obtain:
	\[
	\begin{array}{rl}
		x\we\bn x &\le\bn\on=\bn\bn\ze~~~[\text{Definition~\ref{D:bn-negation}.b and (d)}]\\
		&\le\bn\bn x~~~[\text{\eqref{E:tilde-quasi-idenity} twice}]\\
		&\le\bn\bn\on=\bn\ze\le x\ve\bn x.~~~[\text{\eqref{E:tilde-quasi-idenity} twice, (e) and Definition~\ref{D:bn-negation}.c}]
	\end{array}
	\]
	
	(h): Using \eqref{E:tilde-quasi-idenity}, (a) and (d), we receive:
	\[
	\bn\ze\le x\Longrightarrow
	\bn\on\le\bn x\le\bn\bn\ze=\bn\on.
	\]
	
	(i): In virtue of Definition~\ref{D:bn-negation}.b and (a), we have:
	$\bn x\le(x\ra\bn\on)\we\bn\ze$.
	
	In virtue of Definition~\ref{D:bn-negation}.a, we have:
	$(x\ra\bn\on)\we\bn\bn\on\le\bn x$. Then, we use (e) to get $(x\ra\bn\on)\we\bn\ze\le\bn x$.
	
	(j): With help of (i), we get: 
	\[
	\begin{array}{rl}
	\bn(x\ve y) &=
	(x\ve y\ra\bn\on)\we\bn\ze\\
	&=(x\ra\bn\on)\we(y\ra\bn\on)\we\bn\ze\\
	&=\bn x\we\bn y.
	\end{array}
	\]
	
	(k): First of all, we note that $[\bn\on,\bn\ze]$ is a distributive bounded lattice. Let us take any $x\in[\bn\on,\bn\ze]$. According to (a),
	$\bn x\in[\bn\on,\bn\ze]$. According to Definition~\ref{D:bn-negation}.b and \ref{D:bn-negation}.c, $x\we\bn x=\bn\on$ and $x\ve\bn x=\bn\ze$. Therefore, $\bn x$ is a complement of $x$ in $[\bn\on,\bn\ze]$.
	
	(l) From~\eqref{E:tilde-quasi-idenity}, we derive that $\bn x\in[\bn\on,\bn\ze]$. Then, we apply (k).
\end{proof}

\begin{cor}\label{C:bn-gives-pair}
	Given a \bn-negation, $(\bn\on,\bn\ze)$ is an $\cE$-pair. Hence $\neg\bn\ze=\ze$.
\end{cor}
\begin{proof}
	Indeed, from Proposition~\ref{P:bn-properties}.a, we derive that $\bn\on\le\bn\ze$. And Definition~\ref{D:bn-negation}.d gives us $\bn\ze\ra\bn\on\le\bn\on$. Further, in virtue of Definitions~\ref{D:bn-negation}.c and \ref{D:bn-negation}.b, we obtain that
	$\bn\ze\le x\ve(x\ra\bn\on)$.
	
	The equality $\neg\bn\ze=\ze$ follows from Proposition~\ref{P:negation}
\end{proof}
\begin{cor}
	Given a Heyting algebra $\fA$, two \bn-negations $\bn_{1}$ and $\bn_{2}$ are equal in $\fA$ if and only if $\bn_{1}\on=\bn_{2}\on$.
\end{cor}
\begin{proof}
	Assume that $\bn_{1}\on=\bn_{2}\on$. In view of Proposition~\ref{P:bn-properties}.i, we need to show that  $\bn_{1}\ze=\bn_{2}\ze$. According to Corollary~\ref{C:bn-gives-pair}, both pairs $(\bn_{1}\on,\bn_{1}\ze)$ and
	$(\bn_{2}\on,\bn_{2}\ze)$ belong to $\cE_{\mathfrak{A}}$. In virtue of Proposition~\ref{P:unique}, $\bn_{1}\ze
	=\bn_{2}\ze$. The converse is obvious.
\end{proof}

Now we give an example of how a \bn-negation can be defined in a Heyting algebra.
\begin{prop}\label{P:bn-negation-1}
	Given a Heyting algebra {\fA}, if $(a,\astar)$ is an \cE-pair, then
	the operation
	\[
	\bn x=(x\ra a)\we\astar
	\]
	is a \bn-negation in {\fA} so
	that $a=\bn\on$ and $\astar=\bn\ze$.
\end{prop}
\begin{proof}
	We have to check that the definition of $\bn x$ above satisfies the properties (\text{a})--(\text{d}) of Definition~\ref{D:bn-negation}.
	
	(a): In any Heyting algebra, the following holds:
	\[
	x\ra y\le (y\ra a)\ra(x\ra a)\le
	(y\ra a)\we a^{\ast}\ra(x\ra a)\we a^{\ast}.
	\]
	
	(b): We also have:
	\[
	x\we(x\ra a)\we a^{\ast}=
	x\we a\we a^{\ast}\le a\we a^{\ast}=
	\on\we(\on\ra a)\we a^{\ast}.
	\]
	
	(c): We first note that $a^{\ast}=(\ze\ra a)\we a^{\ast}$. Also, since $a^{\ast}\le x\ve(x\ra a)$, we have: $a^{\ast}\le x\ve(x\ra a)\we a^{\ast}$. Thus $(\ze\ra a)\we a^{\ast}
	\le x\ve(x\ra a)\we a^{\ast}$.
	
	(d): We notice that $a^{\ast}=(\ze\ra a)\we a^{\ast}$ and $a=(\on\ra a)\we a^{\ast}$. Thus the true inequality
	$a^{\ast}\ra a\le a$ implies
	$(\ze\ra a)\we a^{\ast}\ra
	(\on\ra a)\we a^{\ast}\le
	(\on\ra a)\we a^{\ast}$,
	that is $\bn\ze\ra\bn\on\le\bn\on$.
\end{proof}

The last proposition inspires the next definition.
\begin{defi}[$t_{\varepsilon}$-negation, $\varepsilon_{t}$ pair]\label{D:t-negation}
	Given an $\cE$-pair $\varepsilon=(a,a^{\ast})$ in a Heyting algebra, we define a \bn-negation as follows:
	\[
	t_{\varepsilon}(x)=(x\ra a)\we a^{\ast}.
	\]
	On the other hand, given a \bn-negation $t(x)$, we denote $\varepsilon_{t}=(t(\on),t(\ze))$.
\end{defi}

\begin{prop}\label{P:t(x)}
	If $t(x)$ is a \bn-negation, then
	\[
	t_{\varepsilon_{t}}(x)=t(x).
	\]
	If $\varepsilon=(a,a^{\ast})$ is an $\cE$-pair, then
	\[
	\varepsilon_{t_{\varepsilon}}=\varepsilon.
	\]
\end{prop}
\begin{proof}
	According to Definition~\ref{D:t-negation} and Proposition~\ref{P:bn-properties}.i,
	\[
	t_{\varepsilon_{t}}(x)=(x\ra t(\on))\we t(\ze)=t(x).
	\]
	
	Further, according to Definition~\ref{D:t-negation} and Proposition~\ref{P:bn-negation-1},
	\[
	\varepsilon_{t_{\varepsilon}}=(t_{\varepsilon}(\on),t_{\varepsilon}(\ze))=(a,a^{\ast}).
	\]
\end{proof}
\begin{cor}\label{C:bn-pair-isomorphism}
	Given a Heyting algebra $\fA$, there is a one-one correspondence between  $\cE_{\mathfrak{A}}$ and \bn-negations in $\fA$.
\end{cor}

\begin{defi}[\bn-expansion, Heyting reduct, class $K$]\label{D:expansion}
	An algebra $(\fA,\bn)$, where
	$\fA$ is a Heyting algebra with a \bn-negation, is called a \bn-expansion $($a tilde-expansion of {\fA}$)$. Also, we will call {\fA} the Heyting reduct $($or simply reduct$)$ of the \bn-expansion $(\fA,\bn)$. The abstract class of all \bn-expansions is denoted by $K$.
\end{defi}

\begin{prop}\label{P:variety}
	Class $K$ is a variety.
\end{prop}
\emph{Proof}~follows immediately from Definition~\ref{D:bn-negation}.\\

\begin{prop}\label{P:congruences=filters}
	Given a \bn-expansion $\fB=(\fA,\bn)$, there is a one-one correspondence between the
	congruences on $\fB$ and the filters of the Heyting reduct $\fA$.
\end{prop}
\emph{Proof}~follows straightforwardly from
Definition~\ref{D:bn-negation}.a.\\

In the sequel, we will use the last proposition without reference.

So far, talking about one-element enrichment in this section, we introduced $\bn$-negation as a tool to ``materialize'' the enrichabilty of an unspecified element of a Heyting algebra. Now we will be dealing with a particular element of the algebra, which is intended to be enriched.

\begin{defi}[$\tau$-expansion, $\tau\bn$-expansion, classes $K_{\tau}$ and $K^{\ast}_{\tau}$]\label{D:tau-tilde-expasion}
	Let $\fA$ be a Heyting algebra. We enrich the signature of $\fA$ with a nullary operation $\tau$ and call $\fA_{\tau}=(\fA,\tau)$ a $\tau$-expansion of $\fA$. If we know that $\tau$ in the latter is interpreted by $a\in|\fA|$, we will denote this expansion by $\fA_{\tau_{a}}$. The \bn-expansion of $\fA_{\tau}$ that satisfies  the identity $\bn\on=\tau$ is called a $\tau\bn$-expansion $($of $\fA$$)$, in symbols $(\fA_{\tau},\bn)$.
	In $\fA_{\tau}$ and $(\fA_{\tau},\bn)$, $\fA$ is called the \textbf{Heyting reduct} $($or simply \textbf{reduct}$)$ of the former and latter and $\fA_{\tau}$ is the $\tau$-\textbf{Heyting reduct} $($or simply $\tau$-\textbf{reduct}$)$ of the latter. The equational class of all $\tau\bn$-expansions is denoted by $K_{\tau}$. The class of $\tau$-Heyting reducts of the algebras of $K_{\tau}$ is denoted by $K^{\ast}_{\tau}$.
\end{defi}

The following observation is obvious.
\begin{prop}
	The class $K_{\tau}$ is a variety.
\end{prop}

In Section~\ref{S:localization-proof-theoretic-0}, we will see that not only the class $K^{\ast}_{\tau}$ is a variety, but one can prove that the class of all $\tau$-Heyting reducts of any subvariety of $K_{\tau}$ whose equational theory is defined by $\Lantau$-formulas is a subvariety of $K^{\ast}_{\tau}$.

\begin{defi}[packing, relation $\vartriangleleft$]\label{D:packing}
	Suppose $\fA_{\tau}\!\!\suba\!\!\fB_{\tau}$
	and $(\fB_{\tau},\bn)$ is a $\tau\bn$-expansion generated by $|\fA|$. Then, we say that $\fA_{\tau}$ is packed in $\fB_{\tau}$$;$ symbolically
	$\fA_{\tau}\pack\fB_{\tau}$. If $\fA_{\tau}$ is packed in $\fB_{\tau}$, then $(\fA_{\tau},\bn)$ can be regarded as a partial algebra w.r.t. $\bn$ and as such is a relative subalgebra of a $($full$)$ algebra $(\fB_{\tau},\bn)$
	$($in the sense of~\cite{gra79}, \S13$)$; in this case, we also say that $(\fA_{\tau},\bn)$ is packed in 
	$(\fB_{\tau},\bn)$, denoting this by
	$(\fA_{\tau},\bn)\pack(\fB_{\tau},\bn)$.
\end{defi}

\begin{prop}\label{P:generating}
	If $(\fA_{\tau},\bn)\pack(\fB_{\tau},\bn)$,
	then $\fB_{\tau}$ is generated as a $\tau$-expansion by $|\fA|\cup\lbrace\bn\ze\rbrace$. Conversely, if $(\fA_{\tau},\bn)$ is a relative subalgebra of a $\tau\bn$-expansion $(\fB_{\tau},\bn)$ and the latter is generated as a $\tau$-expansion by $|\fA|\cup\lbrace\bn\ze\rbrace$, then 
	$(\fA_{\tau},\bn)\vartriangleleft(\fB_{\tau},\bn)$.
\end{prop}
\begin{proof}
The first part follows straightforwardly by the property Proposition~\ref{P:bn-properties}.i. The second part is obvious.	
\end{proof} 
\begin{prop}\label{P:bn-closeness}	
	Let $(\fB_{\tau},\bn)$ be a $\tau\bn$-expansion and let $(\fA_{\tau},\bn)$ be a relative subalgebra of $(\fB_{\tau},\bn)$. Then the following conditions are equivalent$:$	
{\em	\[
	\begin{array}{rl}
	(\text{a}) & \fA_{\tau}\mbox{ is closed under \bn};\\
	(\text{b}) & \bn a\in\fA_{\tau}\mbox{ and }\bn\bn a\in\fA_{\tau},\mbox{ for some
		$a\in|\fA|$};\\
	(\text{c}) & \bn\ze\in\fA_{\tau}.
	\end{array}
	\]}
\end{prop}
\begin{proof}
	The implication $(\text{a})\Ra(\text{b})$ is obvious.
	
	Then, $(\text{b})\Ra(\text{c})$ follows straightforward from
	Proposition~\ref{P:bn-properties}.c.
	
	Now we prove $(\text{c})\Ra (\text{a})$. Since, by premise, $\bn\on\in\fA_{\tau}$, we use Proposition~\ref{P:bn-properties}.i. 
\end{proof}
\begin{prop}
	Let $\fA_{\tau}\preccurlyeq\fB_{\tau}$ and let $(\fA_{\tau},\bn_{1})$ and $(\fB_{\tau},\bn_{2})$ be $\tau\bn$-expansions. Then the following properties are equivalent:
	\[
	\begin{array}{cl}
	(\emph{a}) &(\fA_{\tau},\bn_{1})\preccurlyeq(\fB_{\tau},\bn_{2});\\
	(\emph{b}) &\bn_{1}\ze=\bn_{2}\ze;\\
	(\emph{c}) &\bn_{1}\tau=\bn_{2}\tau.
	\end{array}
	\]
\end{prop}
\begin{proof}
	The conditional $(\text{a})\Rightarrow(\text{b})$ is obvious. Next, assume that (b) is true. Then, by virtue of Proposition~\ref{P:bn-properties}.e,
	\[
	\bn_{1}\tau=\bn_{1}\bn_{1}\on=\bn_{1}\ze
	=\bn_{2}\ze=\bn_{2}\bn_{2}\on=\bn_{2}\tau;
	\]
	thus (c) is true. Finally, suppose that
	$\bn_{1}\tau=\bn_{2}\tau$. The latter,as above, implies that 
	\[
	\bn_{1}\ze=\bn_{1}\bn_{1}\on=\bn_{1}\tau=\bn_{2}\tau=\bn_{2}\bn_{2}\on=\bn_{2}\ze.
	\]
	Then, we apply Proposition~\ref{P:bn-properties}
\end{proof}

 Let $\fA$ be a Heyting algebra and $a\in|\fA|$. Interpreting a constant $\tau$ as $a$, we get a $\tau$-expansion $\fA_{\tau_{a}}$. Then, we obtain algebra $\de{[\fA_{\tau_{a}}]}$. It is clear that in the latter algebra $(h(\tau_{a}),\de{h(\tau_{a})})$ is an $\cE$-pair (Proposition~\ref{P:delta-h(x)-2}). Thus, by adding of the $\bn$-negation  to $\de{[\fA_{\tau_{a}}]}$, corresponding to this $\cE$-pair, in the way provisioned in Proposition~\ref{P:bn-negation-1}, we obtain a $\tau\bn$-expansion. Moreover, by virtue of Proposition~\ref{P:generating}, $(\fA_{\tau_{a}},\bn)\pack(\de{[\fA_{\tau_{a}}]},\bn)$. We state this conclusion by the following proposition.
\begin{prop}\label{P:packing}
Given a Heyting algebra $\fA$ and $a\in|\fA|$, 	
$(\fA_{\tau_{a}},\bn)\pack(\de{[\fA_{\tau_{a}}]},\bn)$; in other words, $\fA_{\tau_{a}}\pack\de{[\fA_{\tau_{a}}]}$ (in the sense of Definition~\ref{D:packing}).
\end{prop}

\section{One-element enrichment from a proof-theoretic viewpoint}\label{S:localization-proof-theoretic-0}
In this section we prove an analog of Kuznetsov's Theorem (Proposition~\ref{P:equipollence}), where in place of $\square$ the connective $\bn$ is used.\footnote{The proof of Proposition~\ref{P:equipollence} is a modification of our proof of Kuznetsov's Theorem in~\cite{mur15b}.} We need it to derive an analog of Kuznetsov's Corollary 2 mentioned on p.~\pageref{Kuznetsov's-Corollary-2}, which is obtained as Corollary~\ref{C:L(A_tau)=L(B_tau)}.

\subsection{The $\Lantau$-equipollence of two calculi}
\label{S:localization-proof-theoretic}
In this subsection we discus logical systems formulated in languages $\Lantau$ and $\Lantaut$. These languages are extensions of the language $\Lana$ introduced in Section~\ref{S:syntactic}. We obtain $\Lantau$ by adding a nullary connective $\tau$. Then, $\Lantaut$ is the extension of $\Lantau$ by enriching the latter with another unary connective $\sim$. Unspecified formulas of $\Lantau$ will be denoted by symbols $A^{\ast}, B^{\ast},\ldots$ (with or without subscripts) and those of $\Lantaut$ by letters $\alpha,\beta, \gamma$, and $\lambda$ (also with or without subscripts). The formulas of $\Lantaut$ of the form $\sim\!\alpha$ are called $\sim$-\textbf{\textit{formulas}}. We refer to those $\Lantaut$-formulas which do not contain $\sim$ (i.e. are $\Lantau$-formulas)
as $\sim$-\textbf{\textit{free}}. The \textbf{\textit{degree}} of an $\Lantaut$-formula is the number of occurrences of the connective $\sim$ in the formula. Thus all $\Lantau$-formulas have the degree 0. Also, we will be using the following notation:
\[
\true := p\ra p.
\]

Calculus $\Intau$ is defined in the language $\Lantau$, while the calculi $\Intaut$ and $\KMt$ are defined in the language $\Lantaut$. The calculus $\Intau$ is $\Int$ in the language $\Lantau$. The calculus $\Intaut$  is defined by the axioms of $\Int$ in the language $\Lantaut$. The calculus $\KMt$ is $\Intaut$ plus the following formulas:
\[
\begin{array}{rl}
(\text{a}) &\sim p\leftrightarrow(p\rightarrow\tau)\wedge\sim\tau,\\
(\text{b}) & (\sim\tau\rightarrow\tau)\rightarrow\tau,\\
(\text{c}) &\sim\tau\rightarrow(p\vee(p\rightarrow\tau)),\\
(\text{d}) &\tau\rightarrow\sim\tau.
\end{array}
\]
The postulated inference rules of all calculi under consideration are (uniform) substitution and modus ponens.\footnote{The notation, $\KMt$, is justified by the formulas~\eqref{E:KM-axioms}, Proposition~\ref{P:bn-negation-1} and the property~Proposition~\ref{P:bn-properties}.e.}  

Below we will deal with several types of derivation. We distinguish these types as follows.
\begin{itemize}
	\item $\Intau +A^{\ast}\vdash B^{\ast}$ means that there is a derivation in $\Lantau$ of a formula $B^{\ast}$ from axioms of $\Int$ and a formula $A^{\ast}$ as a premise by using substitution of $\Lantau$-formulas and modus ponens..
	\item $\Intaut +\alpha\vdash\beta$ denotes the fact that there is a derivation in $\Lantaut$ of a formula $\beta$ from axioms of $\Int$ and a formula $\alpha$ by using substitution of $\Lantaut$-formulas and modus ponens.
	\item $\KMt+\alpha \vdash\beta$ is to denote that there is a derivation in $\Lantaut$ of $\beta$ from axioms of $\Int$, formulas of the list $\text{(a)}-\text{(d)}$, and $\alpha$.
\end{itemize}
If a derivation $\cD$ supports, say, the claim $\KMt+\alpha \vdash\beta$, we will write
$\cD:\KMt+\alpha \vdash\beta$ and call $\cD$ a $\KMt$-\textbf{\textit{derivation}}. This notation and terminology  apply also to the types of derivation which have been introduced above, as well as to those which will be defined below.
\begin{defi}[refined derivation] 
	A derivation is called refined if all substitutions, if any, apply only to the axioms occurring in the derivation or to a premise.
\end{defi}

It is a well-known fact that if the only postulated inference rules of a calculus are substitution and modus ponens, then any derivation can be transformed to a refined derivation of the same last formula. (Cf.~\cite{sob74,lam79}.)

Suppose $S=\lbrace\alpha_{1},\ldots,\alpha_{n}\rbrace$ is a finite set of $\Lantaut$-formulas. A formula $\sim\!\!\alpha$ is called \textbf{\textit{maximal in}} $S$ if it is a subformula of at least one of the formulas $\alpha_{i}$ and it does not occur in the scope of any occurrence of the connective $\sim$ in any of the formulas $\alpha_{i}$. The set of all maximal formulas of $S$ is denoted by $M(S)$. We also apply this definition, when $S$ is a finite list of $\Lantaut$-formulas. 
\begin{defi}[pure derivation, $\Vdash$]\label{D:pure-derivation}
	A refined {\em$\KMt$}-derivation   {\em$\cD:\KMt+\alpha\vdash\beta$} is called pure if {\em$M(\cD)\subseteq M(\alpha,\beta)$}. The notation {\em$\cD: \KMt+\alpha\Vdash\beta$} reads that $\cD$ is a pure  {\em$\KMt$}-derivation. These definition  and notation apply to {\em$\Intaut$}-derivations as well. 
\end{defi}
Thus,   $\cD:\KMt + A^{\ast}\Vdash B^{\ast}$ if and only if  $\cD:\KMt + A^{\ast}\vdash B^{\ast}$ and $M(\cD)=\emptyset$. Hence, if the first statement is true, then  $\Intau +A^{\ast}\vdash B^{\ast}$. In this section we aim to prove Proposition~\ref{P:equipollence} and Corollary~\ref{C:general-equipollence}.  
We will reach this goal through the following key, though auxiliary, notion.
\begin{defi}[relation $\ll$, set $\mathcal{I}$, root $r_{0}$]\label{D:set-I}
	Let $\mathbb{N}$ be the set of nonnegative numbers. Then we arrange the pairs of
	$\mathbb{N}^{2}$ by the following relation:
	\[
	(x_{1},x_{2})\ll(y_{1},y_{2})\Longleftrightarrow  x_{1}<y_{1},~\text{or}~x_{1}=y_{1}~\text{and}~x_{2}\le y_{2}.
	\]
	We denote $\mathcal{I}=\langle\mathbb{N}^{2},\ll\rangle$. It is clear that $r_{0}=(0,0)$ is the least element of $\mathcal{I}$.  We call $r_{0}$ the root of $\mathcal{I}$.
\end{defi}

A routine check shows  that $\mathcal{I}$ is a poset. 
\begin{prop}
The poset $\mathcal{I}$ satisfies the descending chain condition~{\em\cite{gra78}}.
\end{prop}
\begin{proof}
First we notice that
\begin{equation}\label{E:<<-relation}
(x_{1},x_{2})\ll(y_{1},y_{2})\Longrightarrow x_1\le y_1.
\end{equation}	

Given a pair $s=(x_1,x_2)\in\mathcal{I}$, we call $x_1$,  the \textit{level} of $s$. It is obvious that, given a pair $s$,
there are only finitely many pairs $t$ with $t\ll s$ such that the levels of $s$ and $t$ coincide. Also, according to (\ref{E:<<-relation}), given a pair $s$, each pair $t$ with $t\ll s$ is of a level that is less than or equal to the level of $s$.
\end{proof}

\begin{defi}[down-complete chains in $\mathcal{I}$]
A descending chain in $\mathcal{I}$ is called down-complete if its least element is $r_{0}$.
\end{defi}

\begin{defi}[rank of derivation, $\vdash_{s}$]\label{D:rank}
	We say that a refined derivation $\cD$ is of rank $s\in\mathcal{I}$, where $s=(m,n)$, denoting this fact by {\em$\cD: \KMt+\alpha\vdash_{s}\beta$} $($or by {\em$\cD: \Intaut+\alpha\vdash_{s}\beta$ $)$},  if $M(\cD)\neq\emptyset$, $m$ is the highest degree among the formulas of $M(\cD)$ and $n$ is the number of the formulas of $M(\cD)$ of the degree $m$. If $M(\cD)=\emptyset$ then $s=r_{0}$.
\end{defi}
It is obvious that
\[
\KMt +A^{\ast}\vdash_{r_{0}} B^{\ast} \Longleftrightarrow \KMt +A^{\ast}\Vdash B^{\ast}.
\]

Given formulas $\alpha$, $\beta$ and $\gamma$, we denote by
\[
\alpha[\beta:\gamma]
\]
the result of replacement of all occurrences of $\beta$ in $\alpha$ with $\gamma$.
\begin{lem}\label{P:key-lemma-one}
	Let {\em  $\cD:\KMt +\alpha\vdash_{s} \beta$} be a refined derivation of rank $s\neq r_{0}$. Also,
	suppose a formula  $\sim\!\gamma\in M(\cD)$,  $\gamma\neq\tau$, $\sim\!\gamma$ is not a subformula of $\alpha$ and $\sim\!\gamma$ is of  the highest degree among the formulas of $M(\cD)$ .  Then there are a formula $\delta$ and
	a refined derivation {\em $\cD^{\ast}:\KMt+\alpha\vdash_{t} \beta[\sim\!\gamma:\delta]$} such that$:$
	\begin{itemize}
		\item if $M(\cD)=\lbrace\sim\! A^{\ast}_{1},\ldots,\sim\! A^{\ast}_{n}\rbrace$, for some formulas $A^{\ast}_{1},\ldots,A^{\ast}_{n}$, and $\gamma=A^{\ast}_{i}$,  then  $M(\cD^{\ast})=\lbrace \sim\! A^{\ast}_{1},\ldots,\sim\! A^{\ast}_{i-1},\sim\! A^{\ast}_{i+1},\sim\!\tau\rbrace$, in which case $s=t=(1,n)$$;$
		\item otherwise,   $t\ll s$ and $t\neq s$.
	\end{itemize}
\end{lem}
\begin{proof}
	Let us denote by 
	\[
	\cD: \gamma_{1},\ldots,\gamma_{n}
	\]
	the given derivation. Obviously, $\gamma_{n}=\beta$. Then, we define:
	\[
	\delta=(\gamma\rightarrow\tau)\wedge\sim\!\tau. 
	\]
	We notice that $\delta$ does not contain $\sim\!\gamma$. Further, we define:
	\[
	\gamma_{i}^{\ast}=\gamma_{i}[\sim\!\gamma:\delta].
	\]
	
	We note that $\gamma_{i}^{\ast}$ does not contain $\sim\!\gamma$. If $\gamma\neq A^{\ast}$, for some $A^{\ast}$, then the degree of $\gamma_{i}^{\ast}$ is less than that of $\gamma_{i}$ but greater than or equal to 1; otherwise, the degree of $\gamma_{i}^{\ast}$ equals that of $\gamma_{i}$ and both are equal to $1$. Now we have to consider in more detail what happens in conversion of $\gamma_{i}$ to $\gamma_{i}^{\ast}$. For this we examine the following cases.
	
	Case 1: $\gamma_{i}$ is an instance of an $\Int$-axiom. Then $\gamma_{i}^{\ast}$ is also an instance of the same $\Int$-axiom. 
	
	Case 2: $\gamma_{i}$ is an instance of the axiom (a), that is a formula $\sim\!\lambda\leftrightarrow(\lambda\rightarrow\tau)\wedge\sim\!\tau$. In this case, assume that $\lambda\neq\gamma$.  Then $\gamma_{i}$ does not contain $\sim\!\gamma$ at all and hence $\gamma_{i}^{\ast}=\gamma_{i}$.
	
	Case 3: $\gamma_{i}=\sim\!\gamma\leftrightarrow(\gamma\rightarrow\tau)\wedge\sim\!\tau$. Then $\gamma_{i}^{\ast}=\delta\leftrightarrow\delta$. It is obvious that $\Intaut\Vdash\gamma_{i}^{\ast}$. Let us denote a derivation that supports the last claim by $\cD_{1}$.
	
	Case 4: $\gamma_{i}=\sim\!\tau\rightarrow(\lambda\vee(\lambda\rightarrow\tau))$. Then
	$\gamma_{i}^{\ast}=\sim\!\tau\rightarrow(\lambda[\sim\!\gamma:\delta]\vee(\lambda[\sim\!\gamma:\delta]\rightarrow\tau))$, that is, $\gamma_{i}^{\ast}$ is an instance of the axiom (c) and does not contain $\sim\!\gamma$.
	
	Case 5: $\gamma_{i}$ is an instance of $\alpha$. Then, since $\sim\!\gamma$ is not a subformula of $\alpha$  and $\sim\!\gamma$ is maximal in $\cD$, $\gamma_{i}^{\ast}$ remains to be an instance of $\alpha$.
	
	Case 6: $\gamma_{i}$ is obtained by modus ponens from $\gamma_{k}$ and $\gamma_{l}=\gamma_{k}\rightarrow\gamma_{i}$ for some $k,l<i$. Then $\gamma_{i}^{\ast}$ can be derived from  $\gamma_{k}^{\ast}$ and 
	$\gamma_{l}^{\ast}=\gamma_{k}^{\ast}\rightarrow\gamma_{i}^{\ast}$.

	Further, we define
	\[
	[\gamma_{i}^{\ast}]=
	\begin{cases}
	\begin{array}{cl}
	\gamma_{i}^{\ast} &\text{if $\gamma_{i}^{\ast}$ is obtained in one of the cases 1,2, 4, 5 or 6}\\
	\cD_{1} &\text{if $\gamma_{i}^{\ast}$ is obtained according to case 3.}
	\end{array}
	\end{cases}
	\]
	
	Now we denote:
	\[
	\cD^{\ast}: [\gamma_{1}^{\ast}],\ldots,[\gamma_{n}^{\ast}]. 
	\]
	
	It is clear that $\cD^{\ast}$ is a refined derivation which supports $\KMt+\alpha\vdash\beta[\sim\!\gamma:\delta]$.  Assume that $s=(m,n)$. In the case of the first alternative in the conclusion of the lemma, that is when $s=(1,n)$,  $\delta$ contains only one $\sim$-formula -- $\sim\!\tau$. Regardless of whether $M(\cD)$ contains $\sim\!\tau$ or not, $M(\cD^{\ast})$ will have it. Thus the conclusion of the first alternative is true. Otherwise, $m>1$ and, then, either $t=(m,n-1)$ or $t=(m_{1},n_{1})$ with $1\le m_{1}<m$ and some $n_{1}\ge 1$. We observe that in both cases $t\ll s$ and $t\neq s$.
\end{proof}

\begin{lem}\label{P:key-lemma-two}
	Let {\em  $\cD:\KMt +\alpha\vdash\beta$} be a refined derivation such that $M(\cD)=\lbrace\sim\!\tau\rbrace$. Also, assume that $\sim\!\tau$ is not a subformula of $\alpha$. Then there is a $\sim$-free formula $A^{\ast}$ such that {\em $\KMt+\alpha\Vdash\beta[\sim\!\tau:A^{\ast}]$}.
\end{lem}
\begin{proof}
	Assume that the formulas
	\begin{equation}\label{E:c-axiom-instances}
	\sim\!\tau\rightarrow(\lambda_{1}\vee(\lambda_{1}\rightarrow\tau)),\ldots,
	\sim\!\tau\rightarrow(\lambda_{k}\vee(\lambda_{k}\rightarrow\tau))
	\end{equation}
	are all the instances of the axiom (c)  in the refined derivation $\cD$. Then, we define:
	\[
	A^{\ast}=\begin{cases}
	\begin{array}{cl}
	\bigwedge_{1\le j\le k}(\lambda_{j}\vee(\lambda_{j}\rightarrow\tau))[\sim\!\tau:\true]
	&\text{if the list (\ref{E:c-axiom-instances}) is not empty}\\
	\true &\text{if the list (\ref{E:c-axiom-instances}) is empty}.
	\end{array}
	\end{cases}
	\]
	Thus, if the list (\ref{E:c-axiom-instances})  is nonempty, then we denote:
	\[
	A^{\ast}=\bigwedge_{1\le j\le k}(B^{\ast}_{j}\vee(B^{\ast}_{j}\rightarrow\tau)),
	\]
	for some $\sim$-free formulas $B^{\ast}_{1},\ldots, B^{\ast}_{k}$.  
	
	Further, we denote:
	\[
	\gamma_{i}^{\ast}=\gamma_{i}[\sim\!\tau:A^{\ast}].
	\]
	Now we consider the following cases.
	
	Case 1: $\gamma_{i}$ is an instance of an $\Int$-axiom. Then $\gamma_{i}^{\ast}$ is also an instance of the same $\Int$-axiom.
	
	Case 2: $\gamma_{i}=\sim\!\tau\leftrightarrow(\tau\rightarrow\tau)\wedge\sim\!\tau$. Then $\gamma_{i}^{\ast}=A^{\ast}\leftrightarrow(\tau\ra\tau)\wedge A^{\ast}$. It is obvious that there is a derivation $\cD_{1}:\Intau\Vdash\delta\leftrightarrow(\tau\ra\tau)\wedge A$.
	
	Case 3:  $\gamma_{i}=(\sim\!\tau\ra\tau)\ra\tau$. If (\ref{E:c-axiom-instances}) is empty, then $\gamma_{i}^{\ast}=(\true\ra\tau)\ra\tau$ and hence $\Intau\Vdash \gamma_{i}^{\ast}$. Let us denote a pure derivation that supports the last claim by $\cD_{2}$.
	
	If (\ref{E:c-axiom-instances}) is nonempty, then $\gamma_{i}^{\ast}=
	(\bigwedge_{1\le j\le k}(B^{\ast}_{j}\vee(B^{\ast}_{j}\rightarrow\tau))\ra\tau)\ra\tau$. Then we observe:
	\[
	\begin{array}{l}
	\Intau\Vdash ((B^{\ast}_{j}\ra\tau)\wedge((B^{\ast}_{j}\ra\tau)\ra\tau))\ra\tau, ~\text{that is}\\
	\Intau\Vdash ((B^{\ast}_{j}\vee(B^{\ast}_{j}\rightarrow\tau))\ra\tau)\ra\tau~\text{and hence}\\
	\Intau\Vdash ((B^{\ast}_{j}\vee(B^{\ast}_{j}\rightarrow\tau))\ra\tau)\leftrightarrow\tau;~\text{the latter in turn implies that}\\
	\Intau\Vdash ((B^{\ast}_{1}\vee(B^{\ast}_{1}\rightarrow\tau))\ra ((B^{\ast}_{2}\vee(B^{\ast}_{2}\rightarrow\tau))\ra\tau))\ra\tau,~\text{that is}\\
	\Intau\Vdash(\bigwedge_{1\le j\le 2}(B^{\ast}_{j}\vee(B^{\ast}_{j}\rightarrow\tau))\ra\tau)\ra\tau;~\text{thus, repeating this, we obtain}\\
	\Intau\Vdash(\bigwedge_{1\le j\le k}(B^{\ast}_{j}\vee(B^{\ast}_{j}\rightarrow\tau))\ra\tau)\ra\tau.
	\end{array}
	\]
	We denote by $\cD_{3}$ a pure derivation supporting the last claim.
	
	Case 4: $\gamma^{i}$ is an instance of the axiom (c), that is  $\gamma_{i}=\sim\!\tau\ra(\lambda_{j}\ra(\lambda_{j}\ra\tau))$. Consequently,
	$\gamma_{i}^{\ast}=A^{\ast}
	\ra (\lambda_{j}[\sim\!\tau:A^{\ast}]\vee(\lambda_{j}[\sim\!\tau:A^{\ast}]\ra\tau))$. As it is well-known, (see, e.g.,~\cite{kle52}, {\S} 26)
	\[
	\begin{array}{l}
	\Intau\Vdash A^{\ast}
	\ra (\lambda_{j}[\sim\!\tau:\true]\vee(\lambda_{j}[\sim\!\tau:\true]\ra\tau)),~\text{that is}\\
	\Intau\Vdash \bigwedge_{1\le j\le k}(B^{\ast}_{j}\vee(B^{\ast}_{j}\rightarrow\tau))\ra 
	(B^{\ast}_{j}\vee(B^{\ast}_{j}\rightarrow\tau)).
	\end{array}
	\]
	Therefore, $\Intau\Vdash\gamma_{i}^{\ast}$. We denote a pure derivation supporting the last claim by $\cD_{4}$.
	
	Case 5: $\gamma_{i}=\tau\ra\sim\!\tau$. Then either $\gamma_{i}^{\ast}=
	\tau\ra \true$ or $\gamma_{i}^{\ast}=\tau\ra \bigwedge_{1\le j\le k}(B^{\ast}_{j}\vee(B^{\ast}_{j}\rightarrow\tau))$. Obviously, in both cases $\Intau\Vdash\gamma_{i}^{\ast}$. We denote a pure derivation supporting the last claim by $\cD_{5}$.
	
	Case 6: $\gamma_{i}$ is obtained by modus ponens from  $\gamma_{l}$ and  $\gamma_{l}\ra\gamma_{i}$. Obviously, then $\gamma_{i}^{\ast}$ is obtained by modus ponens from  $\gamma_{l}^{\ast}$ and  $\gamma_{l}^{\ast}\ra\gamma_{i}^{\ast}$.
	
	Now we define:
	\[
	[\gamma_{i}^{\ast}]=\begin{cases}
	\begin{array}{cl}
	\gamma_{i}^{\ast} &\text{if $\gamma_{i}$ falls in Case 1};\\
	\cD_{1 } &\text{if $\gamma_{i}$ falls in Case 2};\\
	\cD_{2} &\text{if $\gamma_{i}$ falls in the first subcase of Case 3};\\
	\cD_{3} &\text{if $\gamma_{i}$ falls in the second subcase of Case 3};\\
	\cD_{4} &\text{if $\gamma_{i}$ falls in Case 4};\\
	\cD_{5} &\text{if $\gamma_{i}$ falls in Case 5}.
	\end{array}
	\end{cases}
	\]
	
	Further, we denote:
	\[
	\cD^{\ast}: [\gamma_{1}^{\ast}],\ldots,[\gamma_{n}^{\ast}].
	\]
	
	It should be clear that $\cD^{\ast}:\KMt+\alpha\Vdash\beta[\sim\!\tau:A^{\ast}]$.
\end{proof}

\begin{prop}\label{P:equipollence}
	The calculi {\em $\KMt$} and {\em $\Intau$} are $\Lantau$-equipollent\emph{;} that is, for any  $\Lantau$-formulas $A^{\ast}$ and $B^{\ast}$,
	{\em 
		\[
		\KMt +A^{\ast}\vdash B^{\ast}\Longleftrightarrow \Intau + A^{\ast}\vdash B^{\ast}.
		\]}
\end{prop}
\begin{proof}
Let $\cD:\KMt +A^{\ast}\vdash_{s} B^{\ast}$ be a refined derivation of rank $s=(m,n)$, where $m,n\ge 1$. Using Lemma~\ref{P:key-lemma-one}, maybe more than one time, we obtain a refine derivation $\cD^{\prime}:\KMt +A^{\ast}\vdash_{t} B^{\ast}$ with $M(\cD^{\prime})=\lbrace\bn\tau\rbrace$. Then, we apply Lemma~\ref{P:key-lemma-two}, to get
$\KMt +A^{\ast}\Vdash B^{\ast}$. The latter means that  $\Intau + A^{\ast}\vdash B^{\ast}$.	
\end{proof}
\begin{cor}\label{C:general-equipollence}
	For any set $\Gamma$ of $\Lantau$-formulas and any $\Lantau$-formula $A^{\ast}$, the following equivalence holds\emph{:}
	{\em 
		\[
		\KMt +\Gamma\vdash A^{\ast}\Longleftrightarrow \Intau + \Gamma\vdash A^{\ast}.
		\]}
\end{cor}
\emph{Proof}~follows immediately from Proposition~\ref{P:equipollence}.\\

\subsection{Completeness of {\em$\KMt$}}
We intend to prove the following.
\begin{prop}\label{P:KMt-completeness}
For any $\Lantaut$-formula $\alpha$, {\em$\KMt\vdash\alpha$} if, and only if, any $\tau\bn$-expansion validates $\alpha$. 
\end{prop}
\begin{proof}
It suffices to show that all proper axioms $(\text{a})-(\text{d})$ of $\KMt$ are valid in any $\tau\bn$-expansion and, conversely, if a $\tau$-expansion with a unary operation $\bn x$ satisfies $(\text{a})-(\text{d})$ of $\KMt$, then it is a $\tau\bn$-expansion.

First we rewrite the proper axioms of $\KMt$ as identities:
\[
\begin{array}{cl}
(\text{a}^{\prime}) &\bn x=(x\ra\tau)\wedge\bn\tau,\\
(\text{b}^{\prime}) &(\bn\tau\ra\tau)\le\tau,\\
(\text{c}^{\prime}) &\bn\tau\le(x\vee(x\ra\tau)),\\
(\text{d}^{\prime}) &\tau\le\bn\tau.
\end{array}
\]

Now let $(\fA_{\tau},\bn)$ be a $\tau\bn$-expansion. Then, we recall, not only the identities $(\text{a})$--$(\text{d})$ of Definition~\ref{D:bn-negation} are true but also $\bn\on=\tau$ (Definition~\ref{D:tau-tilde-expasion}). The latter and Proposition~\ref{P:bn-properties}.e imply that
$\bn\ze=\bn\tau$. And, by virtue of Corollary~\ref{C:bn-gives-pair}, we conclude that
$(\tau,\bn\tau)$ is an $\cE$-pair in $\fA_{\tau}$. This immediately implies that the identities $(\text{b}^{\prime})$--$(\text{d}^{\prime})$ are valid in $(\fA_{\tau},\bn)$. By virtue of Proposition~\ref{P:t(x)},
$(\text{a}^{\prime})$ is also valid.

Next assume that the identities $(\text{a}^{\prime})$--$(\text{d}^{\prime})$ are valid in a $\tau$-expansion $\fA_{\tau}$ with a unary operation $\bn x$. From $(\text{b}^{\prime})$--$(\text{d}^{\prime})$ we derive that $(\tau,\bn\tau)$ is an $\cE$-pair in $\fA_{\tau}$. According to Proposition~\ref{P:bn-negation-1}, $\bn x$ is a $\bn$-negation in $\fA_{\tau}$ and $\tau=\bn\on$; that is $(\fA_{\tau},\bn)$ is a $\tau\bn$-expansion.
\end{proof}

It is clear that the last proposition admits the following generalization.
\begin{cor}\label{C:KM-completeness}
Let $\Lambda\cup\lbrace\alpha\rbrace$ be a set  of $\Lantaut$-formulas. Then 
{\em\[
\KMt+\Lambda\vdash \alpha\Longleftrightarrow\Lambda\models\alpha~(\textit{for all $\tau\bn$-expansions}).	
\]}		
\end{cor}

Also, we obtain the following.
\begin{cor}\label{K^*_tau-variety}
The class $K_{\tau}^{\ast}$ is a variety. Moreover, for any $\Lantau$-formula $A^{\ast}$,
{\em\[
K_{\tau}^{\ast}\models A^{\ast}\Longleftrightarrow\Intau\vdash A^{\ast}.
\]}
\end{cor}
\begin{proof}
	It should be clear that $K_{\tau}^{\ast}$ is closed under formation of direct products, of subalgebras and of homomorphic images.
	
	Now, using Corollary~\ref{C:general-equipollence} with $\Gamma=\emptyset$ and Proposition~\ref{P:KMt-completeness}, we receive the equivalence above.
\end{proof}

\section{Connecting the two viewpoints on one-element enrichment}\label{S:connecting-viewpoints}
 We connect the two viewpoints discussed above in Sections~\ref{S:localization-algebraic} and~\ref{S:localization-proof-theoretic-0} via the following two propositions and corollary. Namely, in this section we aim to show that any $\ta$-expansion can be embedded into such a $\ta$-expansion, where the element corresponding to the constant $\ta$ is enrichable, and both $\ta$-expansions generation one and the same variety, or, equivalently, have the same logic in $\Lantau$.
\begin{prop}\label{P:generated-by-tau-expansions}
Any variety $\cV$ of $\tau$-expansions is generated by the class $\cV\cap K^{\ast}_{\tau}$.
\end{prop}
\begin{proof}
Let
\[
\Gamma=\set{B^{\ast}}{(\forall~\fA_{\tau}^{\prime}\in\cV)
	(\fA_{\tau}^{\prime}\models B^{\ast})}.
\]

Suppose, for some $\fA_{\tau}\in\cV $, $\fA_{\tau}\not\models A^{\ast}$. Then $\Intau+\Gamma\not\vdash A^{\ast}$. By virtue of Corollary~\ref{C:KM-completeness}, $\KMt+\Gamma\not\vdash A^{\ast}$. This implies that there is a $\tau\bn$-expansion $(\fB_{\tau},\bn)$ such that $\fB_{\tau}\models\Gamma$ and $\fB_{\tau}\not\models A^{\ast}$. It remains to notice that 
$\fB_{\tau}\in\cV$.
\end{proof}
\begin{prop}\label{P:L(A_tau)=L(C_tau)}
For any $\tau$-expansion $\fA_{\tau}$, there is a $\tau\bn$-expansion $(\fC_{\tau},\bn)$ such that $\fA_{\tau}\suba\fC_{\tau}$ and $L(\fA_{\tau})=L(\fC_{\tau})$.
\end{prop}
\begin{proof}
Let $\cV$ be the variety of the $\tau$-expansions generated by $\fA_{\tau}$. According to Proposition~\ref{P:generated-by-tau-expansions}, $\fA_{\tau}\in\Ho\Su\Pro(\cV\cap K^{\ast}_{\tau})$. In view of Proposition~\ref{P:congruences=filters}, the $\tau$-expansions have the congruence extension property and hence $\fA_{\tau}\in\Su\Ho\Pro(\cV\cap K^{\ast}_{\tau})$. Now we notice that in each algebra of $\cV\cap K^{\ast}_{\tau}$, the element $\tau$ is enrichable. This will be kept in any direct product of algebras of $\cV\cap K^{\ast}_{\tau}$ and in any homorphic image of the latter, for the first-order formula
\[
\exists x\forall y((\tau\le x)\& (x\ra\tau=\tau)\& (x\le y\vee(y\ra\tau)))
\]
is preserved under formation of direct products and homomorphic images; cf.~\cite{mal73}, Sections 7.4 and 7.5.
Thus there is a $\tau$-expansion $\fC_{\tau}$ such that $\fA_{\tau}\suba\fC_{\tau}$ and $\tau$ is enrichable in $\fC_{\tau}$. Then, by virtue of Proposition~\ref{P:bn-negation-1}, a $\bn$-negation can be defined in $\fC_{\tau}$ so that $\bn\on=\tau$. It remains to notice that $L(\fA_{\tau})=L(\fC_{\tau})$.
\end{proof}

\begin{cor}\label{C:L(A_tau)=L(B_tau)}
For any $\tau$-expansion $\fA_{\tau}$, there is a $\tau\bn$-expansion $(\fB_{\tau},\bn)$ such that $\fA_{\tau}\pack\fB_{\tau}$ and $L(\fA_{\tau})=L(\fB_{\tau})$.
\end{cor}
\begin{proof}
Let $(\fC_{\tau},\bn)$ be a $\tau\bn$-expansion from Proposition~\ref{P:L(A_tau)=L(C_tau)}. Let $(\fB_{\tau},\bn)$ be the subalgebra of $(\fC_{\tau},\bn)$ generated by $|\fA|\cup\lbrace\tau\rbrace$. It remains to notice that $L(\fC_{\tau})\sbe L(\fB_{\tau})\sbe L(\fA_{\tau})$ and, then,  apply Proposition~\ref{P:L(A_tau)=L(C_tau)}.
\end{proof}

Our goal is to prove the following.
\begin{con}\label{conjecture}
Let $\fA$ and $\fB$ be Heyting algebra such that $\fA\suba\fB$. Also, let $a\in|\fA|$ and $(a,\astar)$ be an $\cE$-pair in $\fB$. Then, if
$\fA_{\tau_{a}}\pack\fB_{\tau_{a}}$, then $\fB$ is isomorphic to $\de{[\fA_{a}]}$.	
\end{con}

\section{Properties related to Stone embedding}\label{S:Stone-embedding-properties-main}
First, we define two filters of Heyting algebra, among which we designate one, $F_{a}$. In the sequel, this filter will play a key role.

\subsection{Some filters of Heyting algebra}\label{S:filters}
In this subsection we use \cite{rs70} as a main reference, though employed implicitly.\\

Let us fix a Heyting algebra {\fA} and an element $a\in|\fA|$. Then,
we define:
\[
X_{a}=\set{x\in|\fA|}{x\ra a=a}.
\]

\begin{prop}\label{P:Xa}
	$X_{a}$ is a filter of {\fA}.  Moreover, $X_{a}$ is proper if and only if $a\neq\on$.
\end{prop}
\begin{proof}
	Suppose $x_{1},x_{2}\in X_{a}$, that is, $x_{1}\ra a=a$
	and $x_{2}\ra a=a$. Then we have:
	\[
	\begin{array}{rl}
	x_{1}\we x_{2}\ra a &= x_{1}\ra (x_{2}\ra a)\\
	&= x_{1}\ra a\\
	&= a.
	\end{array}
	\]
	
	Next let $x\ra a=a$, and $y\geq x$. Then we obtain:
	\[
	\begin{array}{rl}
	y\ra a &= y\ra (x\ra a)\\
	&= y\we x\ra a\\
	&= x\ra a\\
	&= a.
	\end{array}
	\]
	
	Finally, it is obvious that $\ze\in X_{a}$ if and only if $a=\on$.
\end{proof}

Now we define
\[
F_{a}=\set{x\ve(x\ra a)}{x\in|\fA|}.
\]

\begin{prop}\label{P:Fa}
	For any Heyting algebra {\fA} and element $a\in|\fA|$, the following
	conditions are equivalent:
	\[
	\begin{array}{cl}
	(\emph{\text{a}}) &y\in F_{a};\\
	(\emph{\text{b}}) &y\ra a\leq a\mbox{ and }a\leq y;\\
	(\emph{\text{c}}) &y\ra a\leq y.
	\end {array}
	\]
\end{prop}
\begin{proof}
	We prove that $(\text{a})\Rightarrow (\text{b})\Rightarrow (\text{c})\Rightarrow (\text{a})$.
	
	$(\text{a})\Rightarrow (\text{b})$: Let $y\in F_{a}$. Then for some $x\in|\fA|$,
	$y=x\ve(x\ra a)$. It is clear that $a\leq y$. Also,
	\[
	\begin{array}{rl}
	y\ra a & =(x\ra a)\we((x\ra a)\ra a)\\
	&= (x\ra a)\we a\\
	&=a.
	\end{array}
	\]
	
	$(\text{b})\Rightarrow (\text{c})$: Obvious, by transitivity of $\leq$.
	
	$(\text{c})\Rightarrow (\text{a})$: Obvious again, for $y\ra a\leq y$ implies
	$y=y\ve (y\ra a)$.
\end{proof}

\begin{cor}\label{C:Fa}
	$F_{a}=X_{a}\cap[a)$ and hence $F_{a}$ is a filter, all elements of which are dense. 
	Also, $F_{a}=\set{y\in|\fA|}{y\ra a\leq y}$. Moreover, $F_{a}$ is proper if {\fA} is nontrivial.\footnote{The fact that $\set{y\in|\fA|}{y\ra a\leq y}$ is a filter was established in~\cite{esakia2006}, Proposition 4.}
\end{cor}

\subsection{Some properties of Stone embedding}\label{S:properties-of-Stone-embedding}
The main references here are~\cite{rs70}, though implicitly, ~\cite{mak72} (see also~\cite{gm05}) and also \cite{mur86}, {\S}1.

Let $\fA\suba\fB$. We define:
\begin{align*}
&\varphi:\SB\longrightarrow\SA:~G\mapsto G\cap|\fA|;\\
&\widetilde{\varphi}:\HA{\SB}\longrightarrow\HA{\SA}:~\mathcal{U}\mapsto\set{\varphi(G)}{G\in\mathcal{U}}.
\end{align*}

We note the following property:
\begin{equation}\label{E:axiliary}
\widetilde{\varphi}(\mathcal{U}\cup\mathcal{V})=\widetilde{\varphi}(\mathcal{U})\cup\widetilde{\varphi}(\mathcal{V}).
\end{equation}

\begin{prop}\label{P:maksimova-lemma-5}
	Let $\fA$ and $\fB$ be Heyting algebras with $\fA\suba\fB$. 
	Also, let $F\in\SA$ and $a\in|\fA|\setminus F$.
	Then there is a filter $G\in\noinB{a}$ such that $F=G\cap|\fA|$.
\end{prop}
\begin{proof}\footnote{The argument employed in this proof is a modification of one ``hidden'' in the proof of Lemma 5 of~\cite{mak72}.} First, we define the filter $[F)_{\mathfrak{B}}$ and note that 
	$[F)_{\mathfrak{B}}\cap|\fA|=F$. Thus the set
	\[
	\Phi:=\set{H}{H~\text{is a $\fB$-filter such that  $H\cap|\fA|=F$}}
	\]
	is nonempty. It is obvious that $\Phi$ satisfies the condition of Zorn's lemma and hence $\Phi$ contains a maximal filter $G$ w.r.t. $\subseteq$. We aim to show that $G\in\SB$.
	
	For contradiction, assume that $x\vee y\in G$, but neither $x\in G$ nor $y\in G$. Next, we define two filters: $H_1:=[G\cup\lbrace x\rbrace)_{\mathfrak{B}}$ and $H_2:=[G\cup\lbrace y\rbrace)_{\mathfrak{B}}$. It is obvious that both $H_1$ and $H_2$ are proper. Now we show that either $H_1\cap|\fA|\subseteq G\cap|\fA|$ or $H_2\cap|\fA|\subseteq G\cap|\fA|$. For contradiction, assume that neither of the last is the case, that is, $H_1\cap|\fA|\not\subseteq G\cap|\fA|$ and $H_2\cap|\fA|\not\subseteq G\cap|\fA|$. This implies that there are $z_1\in H_1\cap|\fA|\setminus G\cap|\fA|$ and $z_2\in H_2\cap|\fA|\setminus G\cap|\fA|$, which yields that $z_1\vee z_2\in H_1\cap H_2\cap|\fA|$. The latter in turn implies that
	$x\rightarrow z_1\vee z_2\in G$ and $y\rightarrow z_1\vee z_2\in G$, that is $z_1\vee z_2\in G$. Then, by definition of $G$, $z_1\vee z_2\in F$ and hence either $z_1\in F$ or $z_2\in F$. In both cases, we get a contradiction, for, if, for example, $z_1\in F$, then $z_1\in G\cap|\fA|$. Thus either $H_1\cap|\fA|\subseteq G\cap|\fA|$ or $H_2\cap|\fA|\subseteq G\cap|\fA|$. Now, let us take the first as true. Then we receive: $F\subseteq H_1\cap|\fA|\subseteq G\cap|\fA|=F$. A contradiction, because $G\subset H_1$ and at the same time $G$ is a maximal filter in $\Phi$. Similarly, we get a contradiction, if we start with the second. Thus $G\in\SB$. Since $a\notin G$, $G\in\noinB{a}$.
\end{proof}

\begin{cor}[comp.~\cite{mak72}, Lemma 5]\label{C:maksimova-lemma-5}
Let $\fA$ and $\fB$ be Heyting algebras with $\fA\suba\fB$. For any filter $F\in\SA$, there is a filter $G\in\SB$ such that $F=G\cap|\fA|$; that is to say, the map $\varphi$ is surjective.\footnote{This property is stated in~\cite{mak72}, Lemma 5, but is not discussed there.}
\end{cor}
\begin{proof} 
We apply Proposition~\ref{P:maksimova-lemma-5} for $a=\ze$.
\end{proof}

\begin{cor}\label{C:axiliary-one}
Let $\fA$ and $\fB$ be Heyting algebras with $\fA\suba\fB$.	For every $x\in|\fA|$, $\widetilde{\varphi}(\hB{x})=\hA{x}$.
\end{cor}
\begin{proof}
For any $x\in|\fA|$, we obtain: 
\[
\begin{array}{rl}
F\in\widetilde{\varphi}(\hB{x}) &\Longleftrightarrow
F=G\cap|\fA|,~\text{for some $G\in\hB{x}$}\\
&\Longleftrightarrow F\in\hA{x}.\quad[\text{in virtue of Corollary~\ref{C:maksimova-lemma-5}}]
\end{array}
\]	
\end{proof}

\begin{prop}[folklore]\label{P:on-prime-filter}
	Let $\fA$ and $\fB$ be Heyting algebras with $\fA\suba\fB$. Also, let
	$F$ be an $\fA$-filter and $a\in|\fA|\setminus F$. Then there is
	a filter $G\in\max\noinB{a}$ such that $F\subseteq G\cap|\fA|$. {\em(Part 1)} In particular, if the algebras $\fA$ and $\fB$ coincide, then there is a filter $G\in\max\noinA{a}$ such that $F\subseteq G$. {\em(Part 2)}
\end{prop}
\begin{proof}
	We define the set:
	\[
	\Phi:=\set{H}{\text{$H$ is a $\fB$-filter, $F\subseteq H\cap|\fA|$ and $a\notin H$}}
	\]
	The set $\Phi$ is nonempty, for the filter $[F)_{\mathfrak{B}}$ belongs to it. Also, it is clear that the set $\Phi$ satisfies the condition of Zorn's lemma. Let $G$ be a maximal filter from $\Phi$.
	By definition, $a\notin G$ and $F\subseteq G\cap|\fA|$. It remains to show that $G\in\SA$. After proving that, we will easily conclude that $G\in\max\noinB{a}$.
	
	For contradiction, assume that for some elements $x$ and $y$ of $\fA$,
	$x\lor y\in G$ but neither $x\in G$ nor $y\in G$. Next, we define
	$H_1:=[G\cup\lbrace x\rbrace)_{\mathfrak{B}}$ and $H_2:=[G\cup\lbrace y\rbrace)_{\mathfrak{B}}$. We aim to show that either $a\notin H_1$ or $a\notin H_2$. For contradiction, we suppose that $a\in H_1$ and $a\in H_2$. This implies that for some elements $u$ and $v$ of $G$, $u\le x\rightarrow a$ and
	$v\le y\rightarrow a$. Both inequalities imply that both
	$x\rightarrow a\in G$ and $y\rightarrow a\in G$ are true and hence, by premise, $a\in G$. A contradiction. Thus either $a\notin H_1$ or $a\notin H_2$.
	Let us take the first as true; that is $a\notin H_1$. Then, since $G\subset H_1$, which implies that $F\subseteq H_{1}\cap|\fA|$, and $a\notin H_1$, $G$ is not maximal in $\Phi$. A contradiction. Similarly, we get a contradiction, if we start with the assumption that $a\notin H_2$. Thus $G\in\SB$ and hence $G\in\noinB{a}$. If $H\in\noinB{a}$ and $G\subseteq H$, then, by definition of $G$, $H=G$. This implies that $G\in\max\noinB{a}$
\end{proof}

\begin{cor}\label{C:on-prime-filter-1}
	Let $\fA$ and $\fB$ be Heyting algebras such that $\fA\suba\fB$ and let $a\in|\fA|$. Then
	$\set{G\cap|\fA|}{G\in\max\noinB{a}}\subseteq\max\noinA{a}$.
\end{cor}
\begin{proof}
	Assume that $G\in\max\noinB{a}$. It is obvious that $G\cap|\fA|\in\noinA{a}$. Let us take any $\fA$-filter $F$ with $G\cap|\fA|\subset F$. For contradiction, assume that $a\notin F$. Then, we define a $\fB$-filter $[G\cup F)_{\mathfrak{B}}$. We note that $G\subset[G\cup F)_{\mathfrak{B}}$. For contradiction, assume that $a\in[G\cup F)_{\mathfrak{B}}$.
	Then there exist elements $u\in G$ and $v\in F$ such that $u\wedge v\le a$. This implies that $v\rightarrow a\in G$ and hence, by premise, that $v\rightarrow a\in F$, that is $a\in F$. A contradiction. Thus $a\notin[G\cup F)_{\mathfrak{B}}$.
	Then, by virtue of Proposition~\ref{P:on-prime-filter} (part 2), there is a filter $H\in\max\noinB{a}$ such that $[G\cup F)_{\mathfrak{B}}\subseteq H$. It is obvious that $G\subset H$. A contradiction. Thus $a\in F$. Hence, $G\cap|\fA|\in\max\noinA{a}$.
\end{proof}
\begin{cor}\label{C:on-prime-filter-2}
	Let $\fA$ and $\fB$ be Heyting algebras such that $\fA\suba\fB$ and let $a\in|\fA|$. For any $\fB$-filter $G$ with $G\cap|\fA|\in\max\noinA{a}$, there is a filter $H\in\max\noinB{a}$ such that $G\subseteq H$ and $H\cap|\fA|=G\cap|\fA|$.
\end{cor}
\begin{proof}
	Let $G$ be a $\fB$-filter and $G\cap|\fA|\in\max\noinA{a}$. The latter in particular implies that $a\notin G$. According to Proposition~\ref{P:on-prime-filter} (part 2), there is a filter $H\in\max\noinB{a}$ such that $G\subseteq H$. The letter implies that
	$G\cap|\fA|\subseteq H\cap|\fA|$. If it were the case that 
	$G\cap|\fA|\subset H\cap|\fA|$, then, by premise, we would have that $a\in H\cap|\fA|$. A contradiction.
\end{proof}
\begin{cor}\label{C:axiliary-two}
Let $\fA$ and $\fB$ be Heyting algebras with $\fA\suba\fB$. Also, let $a\in|\fA|$. Then $\widetilde{\varphi}(\max\noinB{a})=\max\noinA{a}$.
\end{cor}
\begin{proof}
Assume that $G\in\max\noinB{a}$. Then, according to Corollary~\ref{C:on-prime-filter-1}, $\varphi(G)\in\max\noinA{a}$.
Now, we suppose that $F\in\max\noinA{a}$. Let us form the filter $[F)_{\mathfrak{B}}$. We observe that $a\notin[F)_{\mathfrak{B}}$ and
$[F)_{\mathfrak{B}}\cap|\fA|=F$. In virtue of Proposition~\ref{P:on-prime-filter}, there is $G\in\max\noinB{a}$ such that $G\cap|\fA|=F$, that is $F\in\widetilde{\varphi}(\max\noinB{a})$.
\end{proof}

\begin{prop}\label{P:axiliary}
Let $\fA$ and $\fB$ be Heyting algebras with $\fA\suba\fB$. Also, let $a\in|\fA|$ and $(a,a^{\ast})$	be an {\cE}-pair in $\fB$. Then
$\widetilde{\varphi}(\hB{\astar})=\delta\hA{a}$.
\end{prop}
\begin{proof}
	Indeed, we obtain:
\[
\begin{array}{rl}
\widetilde{\varphi}(\hB{\astar}) &= \widetilde{\varphi}(\hB{a}\cup\max\noinB{a})\quad[\text{Proposition~\ref{P:delta-h(x)-1}}]\\
&= \widetilde{\varphi}(\hB{a})\cup\widetilde{\varphi}(\max\noinB{a})
\quad[\text{in virtue of~\eqref{E:axiliary}}]\\
&=\hA{a}\cup\max\noinA{a}\quad[\text{Corollaries~\ref{C:axiliary-one} and~\ref{C:axiliary-two}}]\\
&=\delta\hA{a}\quad[\text{Proposition~\ref{P:delta-h(x)-1}}].
\end{array}	
\]
\end{proof}
\begin{prop}\label{P:F_a-is-included-in-F}
	Let $\fA$ be a Heyting algebra and $a\in|\fA|$. For any filter
	$F\in\noinA{a}$, $F\in\max\noinA{a}$ if, and only if, $F_a\subseteq F$.
\end{prop}
\begin{proof}
	First, we note that if $a=\on$, then the proposition is trivially true. Thus we assume that $a\neq\on$.
	
	Suppose $F\in\noinA{a}$, And, for contradiction, assume that  $x\in F_{a}\setminus F$. We notice that $x\ra a\not\in F$ (Proposition~\ref{P:Fa}) and $x\ve(x\ra a)\not\in F$ (since $F$ is a prime filter). Now we define a filter $G:=[F\cup\lbrace x\ve(x\ra a)\rbrace)_{\mathfrak{A}}$. We aim to show that $a\not\in G$. Indeed, if $a$ were in $G$, then for some $y\in F$, we would have $y\we (x\ve(x\ra a))\le a$, that is
	$(x\ve(x\ra a))\ra (y\ra a)=\on$.  The latter implies that
	$(x\ra (y\ra a))\we((x\ra a)\ra (y\ra a))=\on$, which in turn yields that $y\ra (x\ra a)\in F$. 
	However, the latter immediately implies that $x\ra a\in F$. 
	A contradiction. Thus $a\notin G$. Then, in virtue of Proposition~\ref{P:on-prime-filter}, there is a filter
	$H\max\noinA{a}$ such that $G\subseteq H$. Noticing that $F\subset H$, we get a contradiction once again, which completes the proof of inclusion $F_a\subseteq F$. 
	
	Conversely, assume that $F_a\subseteq F$. For contradiction, we suppose that there is a filter $H\in\noinA{a}$ with $F\subset H$.
	Let $x\in H\setminus F$. Since $F_a\subseteq F$, $x\vee(x\rightarrow a)\in $. And, since the filter $F$ is prime, $x\rightarrow a\in F$. This implies that $x\rightarrow a\in H$ and hence $a\in H$. A contradiction.
\end{proof}

\begin{prop}\label{P:meet-F_a}
	Let $\fA\suba\fB$, $a\in|\fA|$ and $(a,\astar)\in\cE_{\mathfrak{B}}$. If $\bigwedge\! F_a$ exists in $\fB$, where $F_a:=\set{x\in|\fA|}{x\rightarrow a\le x}$, then
	$\bigwedge\! F_a=\astar$ in $\fB$.
\end{prop}
\begin{proof}
Let us denote $u:=\bigwedge\! F_a$. We note that $\astar\le x$, for every $x\in F_a$ (Section~\ref{S:filters}). Hence $\astar\le u$. For contradiction,  assume that $u\not\le\astar$. Then $a\notin[u)_{\mathfrak{B}}$. In virtue of Propositional~\ref{P:on-prime-filter} (part 2), there is a filter $G\in\max\noinB{a}$ such that $[u)_{\mathfrak{B}}\subseteq G$. According to Proposition~\ref{P:F_a-is-included-in-F}, we obtain that
$[\astar)_{\mathfrak{B}})\subseteq G$, which implies that, on the one hand $u\vee\astar$ is a lower bound of $F_a$, and, on the other, $u<u\vee \astar$. 
\end{proof}

\begin{cor}\label{C:meet-F_a-1}
	Let $\fA\suba\fB$, $a\in|\fA|$ and $(a,\astar)\in\cE_{\mathfrak{B}}$. Then the equality
	$\bigwedge\set{\hB{x}}{x\in F_a}=\hB{\astar}$ in $h_{\mathfrak{B}}[\fB]$, where
	$F_a:=\set{x\in|\fA|}{x\rightarrow a\le x}$.
\end{cor}
\begin{proof}
We note that $h_{\mathfrak{B}}[\fB]\suba\HA{\SB}$ and $(\hB{a},\hB{\astar})$ is an $\cE$-pair in $\HA{\SB}$
(Proposition~\ref{P:delta-h(x)-2}). Since $\bigcap\set{\hB{x}}{x\in F_a}$ is the greatest lower bound of $\set{\hB{x}}{x\in F_a}$ in $\HA{\SB}$, then the equality  is true $\bigwedge\set{\hB{x}}{x\in F_a}=\hB{\astar}$ in $\HA{\SB}$ and also in $h_{\mathfrak{B}}[\fB]$.
\end{proof}

\begin{cor}\label{C:meet-F_a-2}
		Let $\fA$ be a Heyting algebra and $a\in|\fA|$. Then
	$\bigwedge\set{\hA{x}}{x\in F_{a}}$ exists in $\de{[\fA_{a}]}$ and
	the equality
	$\de\hA{a}=\bigwedge\set{\hA{x}}{x\in F_{a}}$,  where
	$F_a:=\set{x\in|\fA|}{x\rightarrow a\le x}$, is true in $\delta[\fA_{a}]$.
\end{cor}
\begin{proof}
We first derive from the given the following: $h_{\mathfrak{A}}[\fA]\suba\HA{\SA}$, $\hA{a}\in h_{\mathfrak{A}}[\fA]$ and $(\hA{a},\de{\hA{a}})$ is an $\cE$-pair in $\HA{\SA}$ (Proposition~\ref{P:delta-h(x)-2}); also,
$\bigwedge\set{\hA{x}}{x\in F_{a}}=\bigcap\set{\hA{x}}{x\in F_{a}}$
in $\HA{\SA}$. Then, in virtue of Proposition~\ref{P:meet-F_a},
the equality $\bigwedge\set{\hA{x}}{x\in F_{a}}=\de{\hA{a}}$ holds
in $\HA{\SA}$. Since $\de{[\fA_{a}]}\suba\HA{\SA}$, the last equality also holds in $\de{[\fA_{a}]}$.
\end{proof}

\section{Completing the proof of Theorem~\ref{T:main-theorem}}\label{S:completing-proof}

We aim to prove Conjecture~\ref{conjecture}, from which Theorem~\ref{T:main-theorem} will follow straightforwardly.\\

Let $\fA$ and $\fB$ be Heyting algebras such that $\fA\suba\fB$. Gradually, we will be adding more conditions.\\

First, we observe that
\begin{equation}\label{E:tilde-phi-surjective}
\textit{the map $\widetilde{\varphi}$ is surjective.}
\end{equation}

Indeed, this follows from that the map $\varphi$ is surjective (Corollary~\ref{C:maksimova-lemma-5}).\\

Next, we remind the reader that the map $\varphi^{-1}$ is an embedding of $\HA{\SA}$ into $\HA{\SB}$; cf.~\cite{mak72}, Lemma 2.\\

The following observation is obvious: For any $\mathcal{U}\in\HA{\SB}$,
\begin{equation}\label{E:phi^{-1}(tilde-phi(U))}
\mathcal{U}\subseteq\varphi^{-1}(\widetilde{\varphi}(\mathcal{U})).
\end{equation}

Also, it is easy to see that
\begin{equation}\label{E:phi^{-1}(h_A(x)=h_B(x)}
\varphi^{-1}(h_{\mathfrak{A}}(x))=h_{\mathfrak{B}}(x),~\textit{for any $x\in|\fA|$}.
\end{equation}

Indeed, we have:
\[
G\in\varphi^{-1}(h_{\mathfrak{A}}(x))\Longleftrightarrow
G\cap|\fA|\in h_{\mathfrak{A}}(x)\Longleftrightarrow G\in h_{\mathfrak{B}}(x).
\]

Now, assume that $a\in|\fA|$ and $(a,\astar)$ is an $\cE$-pair in $\fB$. Then
\begin{equation}\label{E:h_B(astar)-included-ph^{-1}(delta-h_A(a))}
h_{\mathfrak{B}}(\astar)\subseteq\varphi^{-1}(\de{h_{\mathfrak{A}}(a)}).
\end{equation}

Indeed, with the help of~\eqref{E:phi^{-1}(tilde-phi(U))} and Proposition~\ref{P:axiliary}, we obtain:
\[
h_{\mathfrak{B}}(\astar)\subseteq\varphi^{-1}(\widetilde{\varphi}(h_{\mathfrak{B}}(\astar)))=\varphi^{-1}(\de{h_{\mathfrak{A}}(a)}).
\]

Finally, since $\de{h_{\mathfrak{A}}(a)}$ is a lower bound of the set $\set{h_{\mathfrak{A}}(x)}{x\in F_{a}}$ in $\HA{\SA}$ (Proposition~\ref{P:delta-h(x)-2}), using~\eqref{E:phi^{-1}(h_A(x)=h_B(x)}, we conclude that $\varphi^{-1}(\de{h_{\mathfrak{A}}(a)})$ is a lower bound of the set
$\set{h_{\mathfrak{B}}(x)}{x\in F_{a}}$ in $\HA{\SB}$. And, in virtue of Corollary~\ref{C:meet-F_a-1}, we obtain that 
\begin{equation}\label{E:last}
\varphi^{-1}(\de{h_{\mathfrak{A}}(a)})=h_{\mathfrak{B}}(\astar).
\end{equation}

Further, with the help of \eqref{E:phi^{-1}(h_A(x)=h_B(x)} and \eqref{E:last}, we derive the statement of Conjecture~\ref{conjecture}.\\

At last, using Conjecture~\ref{conjecture} and Corollary~\ref{C:L(A_tau)=L(B_tau)}, we obtain that $L(\fA)=L(\delta[\fA_{a}])$. Thus the condition of Corollary~\ref{C:reducibility-4} is satisfied and hence Theorem~\ref{T:main-theorem} is proven.

\section{Discussion}\label{S:conclusion}
We have proved that, given a Heyting algebra $\fA$, the algebra $\direct$ defines the same equational class as $\fA$ does. (Theorem~\ref{T:main-theorem}) In addition, $\direct$ can inherit some algebraic and cardinality properties, if $\fA$ has them. (Propositions~\ref{P:countable} and~\ref{P:sub-irr}) In view of all these properties, we formulate several open questions, dividing them into two problem sets.\\

\noindent{\textsc{Problem set 1}}\\
(a)~\textit{Is $\direct$ finitely subdirectly irreducible, providing that $\fA$ is?}~\footnote{For definition, see e.g.~\cite{gra79}.}\\
(b)~\textit{Is $\direct$ a double Heyting 
	 algebra (alias bi-Heyting algebra), providing that $\fA$ is?}~\footnote{Double Heyting algebras were studied perhaps for the first time in the doctoral dissertation of C. Rauszer, named there \emph{semi-Boolean algebras}; cf.~\cite{rauszer1971}. Interest in these algebras became especially evident after S. Ghilardi proved in~\cite{ghilardi1992} that every finitely generated free Heyting algebra is a bi-Heyting algebra.}\\
(c)~\textit{Is $\direct$ projective $($weakly projective$)$, providing that $\fA$ is?}~\footnote{See definitions e.g. in~\cite{gra79}.}\\
(d)~\textit{Is $\direct$ finitely approximable, providing that $\fA$ is?}~\footnote{See the definition in~\cite{mal73}, p. 60.}\\

Questions like the ones above can be multiplied; we chose only a few.  \\
 
The other category of questions is related to elementary properties which may be preserved in $\direct$.\\

\noindent{\textsc{Problem set 2}}\\
(a)~\textit{Do $\direct$ and $\fA$ have the same quasi-equational theory?}\\
(b)~\textit{Which elementary properties are preserved in $\direct$?}



\begin{thebibliography}{10}
	\bibitem[Church(1956)]{church1956}
	Church, A. (1956) {\em Introduction to {M}athematical {L}ogic},
	Princeton Univ. Press, Princeton. 
	
	
	\bibitem[Esakia(2006)]{esakia2006}
	Esakia, L. (2006) The modalized {H}eyting calculus: a conservative modal
	extension of the intuitionistic logic. {\em J. Appl. Non-Classical Logics}, 16(3--4), 349--366.
	
	\bibitem[Gabbay and Maksimova(2005)]{gm05}
	Gabbay, D. and Maksimova, L. (2005)
	{\em Interpolation and {D}efinability: {M}odal and {I}ntuitionistic
		{L}ogics}, volume~46 of {\em Oxford Logic Guides}.
    The Clarendon Press Oxford University Press, Oxford.
	
	\bibitem[Ghilardi(1992)]{ghilardi1992}
	Ghilardi, S. (1992)
	Free {H}eyting algebras as bi-{H}eyting algebras.
	{\em C. R. Math. Rep. Acad. Sci. Canada}, 6:240--244.
	
	
	\bibitem[Gorbunov(1998)]{gor98}
	Gorbunov, V. (1998)
	{\em Algebraic {T}heory of {Q}uasivarieties}.
	Siberian School of Algebra and Logic. Consultants Bureau, New York.
	
	\bibitem[Gr{\"a}tzer(1978)]{gra78}
	Gr{\"a}tzer, G. (1978)
	{\em {G}eneral {L}attice {T}heory}, volume~75 of {\em Pure and
		Applied Mathematics}.
	 Academic Press, Inc. [Harcourt Brace Jovanovich, Publishers], New
	York-London.
	
	\bibitem[Gr{\"a}tzer(1979)]{gra79}
	Gr{\"a}tzer, G. (1979)
	 {\em Universal {A}lgebra}.
	 Springer-Verlag, New York, second edition.
	
	
	\bibitem[Kleene(1956)]{kle52}
	Kleene, S.
	{\em Introduction to {M}etamathematics}.
	 D. Van Nostrand Co., Inc., New York, N. Y., 1952.
	
	\bibitem[Kuznetsov(1985)]{kuz85b}
	Kuznetsov, A. (1985)
	The proof-intuitionistic propositional calculus.
	{\em Soviet Mathematics-Doklady}, 32(1):27--30.
	
	\bibitem[Kuznetsov and Muravitsky(1986)]{km86}
	Kuznetsov, A. and Muravitsky, A. (1986)
	\newblock On superintuitionistic logics as fragments of proof logic extensions.
	{\em Studia Logica}, 45(1):77--99.
	
	\bibitem[Lambros(1979)]{lam79}
	Lambros, C. (1979)
	A shortened proof of {S}oboci\'nski's theorem concerning a restricted
	rule of substitution in the field of propositional calculi.
	{\em Notre Dame Journal of Formal Logic}, 20(1):112--114.
	
	\bibitem[Maksimova(1972)]{mak72}
	Maksimova, L. (1972)
	Pretabular superintuitionistic logics.
	{\em Algebra and Logika}, 11:308--314.
	
	\bibitem[ Mal'cev(1973)]{mal73}
     Mal'cev, A. (1973)
	{\em Algebraic {S}ystems}.
	Akademie-Verlag, Berlin.
	 Posthumous edition, edited by D. Smirnov and M. Ta{\u\i}clin,
	Translated from the Russian by B. D. Seckler and A. P. Doohovskoy.
	
	\bibitem[Muravitsky(1985)]{mur85}
	Muravitsky, A. (1985)
	Correspondence of proof-intuitionistic logic extensions to
	proof-logic extensions.
	{\em Soviet Mathematics-Doklady}, 31(2):345--348.
	
	\bibitem[Muravitsky(1988)]{mur86}
	Muravitsky, A. (1988)
	Algebraic proof of the separation property for the
	proof-intuitionistic calculus.
	{\em Mathematics of the USSR-Sbornik}, 59(2):397--406.
	
	\bibitem[Muravitsky(1990)]{mur90}
	Muravitsky, A. (1990)
	Magari and $\Delta$-pseudo-{B}oolean algebras.
	{\em Siberian Mathematical Journal}, 31(4):623--628.
		
	\bibitem[Muravitsky(2008)]{mur08}
	Muravitsky, A. (2008)
	The contribution of {A}. {V}. {K}uznetsov to the theory of modal
	systems and structures.
	{\em Logic and Logical Philosophy}, 17(1-2):41--58.
	
	\bibitem[Muravitsky(2014)]{mur14a}
	Muravitsky, A. (2014)
	Logic {K}{M}: a biography.
	In G.~{B}ezhanishvili, editor, {\em {L}eo {E}sakia on {D}uality in
		{M}odal and {I}ntuitionistic {L}ogics}, pages 155--185. Springer.
	
	\bibitem[Muravitsky(2015a)]{mur15a}
	Muravitsky, A. (2015a)
	Interconnection of the lattices of extensions of four logics.
	{\em Log. Univers.}, 11:253--281.
	
	
	\bibitem[Muravitsky(2015b)]{mur15b}
	Muravitsky, A. (2015b)
	On the equipollence of the calculi {I}nt and {K}{M}.
	{URL}~\url{http://arXiv.or/pdf/1702.00054}.
	
	\bibitem[Rasiowa and Sikorski(1970)]{rs70}
	Rasiowa, H. and Sikorski, R. (1970)
	{\em The {M}athematics of {M}etamathematics}.
	PWN---Polish Scientific Publishers, Warsaw, third edition, 
	Monografie Matematyczne, Tom 41.
	
	\bibitem[Rauszer(1971/1972)]{rauszer1971}
	Rauszer, C. (1971/1972) 
	Representation theorem for semi-Boolean algebras. I, II. {\em Bull. Acad. Polon. Sci. S\'{e}r. Sci. Math. Astronom. Phys.}, 
	19:881--887; ibid, 19:889--892.
	
	\bibitem[Soboci{\'n}ski(1974)]{sob74}
	Soboci{\'n}ski, B. (1974)
	A theorem concerning a restricted rule of substitution in the field
	of propositional calculi. {I}, {II}.
	{\em Notre Dame Journal of Formal Logic}, 15:465--476; ibid. 15
	(1974), 589--597.
	
\end{thebibliography}
\end{document}